%% file: KdV_34_v3.tex
\documentclass[final,reqno,onefignum,onetabnum]{siamart190516}
\usepackage{braket,amsfonts}

\usepackage{array}

\usepackage[caption=false]{subfig}
\usepackage{pgfplots}

\newsiamthm{claim}{Claim}
\newsiamremark{remark}{Remark}
\newsiamremark{hypothesis}{Hypothesis}
\crefname{hypothesis}{Hypothesis}{Hypotheses}

\usepackage{algorithmic}

\usepackage{graphicx,epstopdf}

\Crefname{ALC@unique}{Line}{Lines}

\usepackage{amsopn}

\usepackage{xspace}
\usepackage{bold-extra}
\usepackage[most]{tcolorbox}

\colorlet{texcscolor}{blue!50!black}
\colorlet{texemcolor}{red!70!black}
\colorlet{texpreamble}{red!70!black}
\colorlet{codebackground}{black!25!white!25}

\lstdefinestyle{siamlatex}{%
  style=tcblatex,
  texcsstyle=*\color{texcscolor},
  texcsstyle=[2]\color{texemcolor},
  keywordstyle=[2]\color{texemcolor},
  moretexcs={cref,Cref,maketitle,mathcal,text,headers,email,url},
}

\tcbset{%
  colframe=black!75!white!75,
  coltitle=white,
  colback=codebackground, 
  colbacklower=white, 
  fonttitle=\bfseries,
  arc=0pt,outer arc=0pt,
  top=1pt,bottom=1pt,left=1mm,right=1mm,middle=1mm,boxsep=1mm,
  leftrule=0.3mm,rightrule=0.3mm,toprule=0.3mm,bottomrule=0.3mm,
  listing options={style=siamlatex}
}

\newtcblisting[use counter=example]{example}[2][]{%
  title={Example~\thetcbcounter: #2},#1}

\newtcbinputlisting[use counter=example]{\examplefile}[3][]{%
  title={Example~\thetcbcounter: #2},listing file={#3},#1}

\DeclareTotalTCBox{\code}{ v O{} }
{ 
  fontupper=\ttfamily\color{black},
  nobeforeafter,
  tcbox raise base,
  colback=codebackground,colframe=white,
  top=0pt,bottom=0pt,left=0mm,right=0mm,
  leftrule=0pt,rightrule=0pt,toprule=0mm,bottomrule=0mm,
  boxsep=0.5mm,
  #2}{#1}

\patchcmd\newpage{\vfil}{}{}{}
\begin{tcbverbatimwrite}{tmp_\jobname_header.tex}
\title{Notes on a high order fully discrete scheme for the Korteweg-de vries equation with a time-stepping procedure of Runge-Kutta-composition type
}

\author{Vassilios A. Dougalis\thanks{Mathematics Department, University of Athens, 15784
Zographou, Greece and
Institute of Applied and Computational
Mathematics, FO.R.T.H., 71110 Heraklion, Greece
(\email{doug@math.uoa.gr}).}
\and Angel Dur\'an\thanks{Applied Mathematics Department,  University of
Valladolid, 47011 Valladolid, Spain and Institute of Mathematics of the University of Valladolid (IMUVA) Paseo de Belen S/N 47011 Valladolid, Spain (\email{angel@mac.uva.es})}}

\headers{A high order fully discrete scheme for the KdV}{V. A. Dougalis, and A. Dur\'an}
\end{tcbverbatimwrite}
\input{tmp_\jobname_header.tex}

\begin{document}
\maketitle

\begin{tcbverbatimwrite}{tmp_\jobname_abstract.tex}
\begin{abstract}
We consider the periodic initial-value problem for the Korteweg-de Vries equation that we discretize in space by a spectral Fourier-Galerkin method and in time by an implicit, high order, Runge-Kutta scheme of composition type based on the implicit midpoint rule. We prove $L^{2}$ error estimates for the resulting semidiscrete and the fully discrete approximations.
\end{abstract}
\end{tcbverbatimwrite}
\input{tmp_\jobname_abstract.tex}

\section{Introduction}
In this paper we consider the periodic initial-value problem (ivp) for the Korteweg-de Vries (KdV) equation
%
%
%
%
%
\begin{eqnarray}
&&u_{t}+uu_{x}+u_{xxx}=0,\quad x\in [-\pi,\pi], \; 0\leq t\leq T,\label{dd12}\\
&&u(x,0)=u_{0}(x),\quad x\in [-\pi,\pi],\nonumber
\end{eqnarray}
where $u_{0}$ is a smooth, $2\pi-$periodic, real-valued function. The KdV is one of the simplest nonlinear partial differential equations (pde) modelling one-way propagation in one space dimension of long waves in which the nonlinear term (here given by $uu_{x}$) and the linear dispersive term (modelled by $u_{xxx}$) are suitably balanced. It has been studied extensively and has a rich mathematical theory. In the case of (\ref{dd12}) it is well known for example, cf. e.~g. \cite{Te,BS,BonaDKM1995}, that if $u_{0}\in H^{\mu}$ for $\mu\geq 2$, where $H^{\mu}$ is the $L^{2}$-based Sobolev space of order $\mu$ of periodic functions on $[-\pi,\pi]$, then for any $T>0$, (\ref{dd12}) has  a unique solution in $C(0,T;H^{\mu})$, that also belongs to $C^{k}(0,T;H^{\mu-3k})$ for $k\leq \left[\frac{\mu+2}{3}\right]$. (Here $C(0,T;X)$ is the space of continuous maps $u:[0,T]\rightarrow X$, where $X$ is a Banach space, and $C^{k}(0,T;X)$ is the space of $X$-valued functions defined on $[0,T]$ that are $k$ times continuously differentiable.)

We will analyze a high-order fully discrete, conservative numerical method for (\ref{dd12}). The scheme consists of a spectral Fourier-Galerkin discretization of the pde in the spatial variable, coupled with a high-order, diagonally implicit Runge-Kutta (RK) time-stepping scheme of {\it composition type} based on the Implicit Midpoint Rule. Although the analysis is done in the case of the model problem (\ref{dd12}), the main ideas and techniques behind the derivation of the error estimates may be used to establish analogous results for more complicated, $L^{2}$-conservative periodic ivp's for one-way nonlinear dispersive wave pde's with more general nonlinearities and linear dispersive terms and may also prove useful in analyzing temporal discretizations by more general composition-type RK schemes.

Among the many available $L^{2}$-conservative spatial discretizations for (\ref{dd12}) (cf. e.~g. the references of \cite{BakerDK1983} and \cite{BCKX}), we chose, for reasons of simplicity, the spectral Fourier-Galerkin method. This semidiscretization conserves the first three invariants of the KdV and is straightforward to analyze; for rigorous error estimates for the semidiscrete problem cf. e.~g. \cite{MQ,DM,Ka} and their references. In the first two of these papers one may find proofs of $L^{2}$ error bounds of spectral accuracy, whose rates of convergence depend on the smoothness of the initial value $u_{0}$. Specifically, if $N$ is the order of the trigonometric polynomials used in the Fourier basis, it is shown in \cite{MQ} by an energy method that if $u_{0}\in H^{\mu}, \mu\geq 2$, then the $L^{2}$ error of the semidiscrete problem is of $O(N^{1-\mu})$. In \cite{DM} the estimate is improved to $O(N^{-\mu})$ if $\mu\geq 3$, in fact for the generalized KdV equation. In order to obtain this optimal-order result the authors of \cite{DM} compare the semidiscrete approximation to the third-order projection of \cite{Wa}, and the proof is accordingly more complicated. In \cite{Ka} a result of different kind is proved: Specifically, if $u_{0}$ is analytic in a strip about the real axis, then the $L^{2}$ error bound is of $O(e^{-\sigma N})$, where $\sigma$ is a constant depending on $T$; the proof relies on analyticity results in \cite{BG}.

In this paper, since we will be primarily concerned with establishing error estimates for our fully discrete scheme, we give in Section \ref{sec3} a simplified proof of the error of the semidiscretization with an $L^{2}$ error bound of $O(N^{1-\mu})$, provided $\mu\geq 2$; the method of proof differs from that of \cite{MQ}. An important property of the semidiscrete spectral approximation is that its temporal derivatives are bounded, uniformly with respect to $T$ and $N$, in the Sobolev space norms, provided $u_{0}$ is sufficiently smooth; cf. Proposition \ref{propo32}. This property simplifies considerably estimating the errors of the full discretization.

An efficient time-stepping procedure for a conservative spatially discrete method for (\ref{dd12}), such as the one considered here, shoulñd be chosen so that the resulting fully discrete scheme has the following properties:
\begin{itemize}
\item It is $L^{2}$-conservative, preferably symplectic: These properties will give the scheme the chance to simulate accurately properties of the solution of the KdV that depend on the balance of dispersive and nonlinear terms, such as the propagation of solitary waves with constant speed and shape and their asymptotic stability properties, for example, the resolution of general initial profiles into sequences of solitary waves plus dispersive tails, their interactions, etc. A dissipative scheme will not reproduce accurately such properties as time increases.
\item It is convergent, at most under a weak mesh condition.
\item It is of high temporal accuracy, in order to take advantage of the high accuracy in space.
\item It may be easily implemented.
\end{itemize}

The class of implicit Runge-Kutta methods includes schemes that fulfill the above requirements. An example is the family of Gauss-Legendre collocation schemes. It is well known, cf. e.~g. \cite{HairerLW2004} and its references, that the $q$-stage Gauss-Legendre scheme has order of accuracy equal to $2q$, is B-stable and symplectic. These schemes have been used for the temporal discretization of many nonlinear dispersive wave pde's that give rise to stiff semidiscrete systems. Their convergence was analyzed in \cite{BonaDKM1995} in the case of the periodic ivp for the generalized KdV equation, discretized in space by the Galerkin finite element method with smooth periodic splines.

In the paper at hand for the temporal discretization we wil use implicit, symplectic RK schemes of composition type, whose general step is constructed as the composition of $s$ steps, of length $b_{i}k, 1 \leq i\leq s$, (where $k$ is the basic time step), of the Implicit Midpoint Rule, cf. e.~g. \cite{Yoshida1990,FrutosS1992,S-SA,HairerLW2004,SanzSC1994}. For general RK-composition methods we refer the reader to \cite{HairerLW2004} and its references. The particular scheme corresponding to $s=3$, of fourth-order temporal accuracy, was used by de Frutos and Sanz-Serna, \cite{FrutosS1992}, to integrate the ivp (\ref{dd12}) for the KdV, discretized in space by finite element and spectral methods. It was also used in \cite{DDM2019} (see also the arxiv version of the paper) for long time computations in a study of the evolution and stability of solitary waves of the generalized Benjamin equation (see Section \ref{sec6} in the sequel), discretized in space by a spectral method. It should be pointed out that the schemes in this class are not A-stable, since some of the $b_{i}$ are not positive and the attendant rational approximations to $e^{z}$ have poles in the left half of the complex plane. However, for a conservative problem like (\ref{dd12}) the scheme, being symplectic, is unconditionally $L^{2}$-conservative and convergent under a Courant number stability restriction, as will be proved in Theorem \ref{Theo51} in this paper. Let us also remark that symplectic schemes have other well-known properties related to their long-time fidelity to solutions of a problem like (1.1). For example, since the spectral semidiscretization of (1.1), when implemented in the Fourier collocation form, leads to a Hamiltonian system of ordinary differential equations (ode’s) for the semidiscrete solution at the collocation points (the proof for the KdV case follows along similar lines to those in  \cite{Cano} for the nonlinear wave equation and the nonlinear Schr\"{o}dinger equation), the property of symplecticity, \cite{SanzSC1994,HairerLW2004}, ensures the virtual preservation of the Hamiltonian, in the sense that the error in this quantity decreases exponentially for long times when a symplectic method is used.

In section \ref{sec41} we review the error estimate for the fully discrete IMR-spectral scheme, while in \ref{sec42} we present the RK-composition scheme under study in the context of ode's. In section \ref{sec43} we consider the general $s$-stage fully discrete scheme and establish the existence of its solutions, it $L^{2}$- conservation property, and state, under general hypotheses, a result on the uniqueness of solutions. In section \ref{sec44} we study the local temporal error of the time-stepping scheme with $s=3$ stages (of fourth order of accuracy), applied to the semidiscrete system. Assuming that the solution of (\ref{dd12}) is sufficiently smooth and that $k=O(N^{-1})$ we prove in Proposition \ref{propo43} that the local temporal error is of $O(k^{5})$ in $L^{2}$. The result is achieved by computing the asymptotic expansions in powers of $k$, up to $O(k^{5})$ terms, of the intermediate steps of the local error about the points $\tau^{n,i}=t^{n}+(b_{1}+\cdots+b_{i})k$ in terms of the semidiscrete solution and its partial derivatives. We compute the coefficients of these asymptotic expansions, estimate their residuals, and substitute them in the final stage of the local error equation, whereupon, after cancellation, there emerges the $O(k^{5})$ local error. Thus, the overall plan of the proof resembles that adopted in the case of other implicit, high-order RK schemes for the KdV and its generalized version in \cite{DK,BonaDKM1995,K-McK}, for the nonlinear Schr\"{o}dinger equation in \cite{KAD}, and for the explicit, $(4,4)$ \lq classical\rq\ Runge-Kutta scheme for the system of Shallow Water equations in \cite{ADK}. With the exception of \cite{K-McK}, where only the temporal discretization of the pde was considered, in the other papers cited above the spatial discretization was effected by Galerkin finite element methods and the stages of the local temporal error were computed in terms of continuous in time finite element approximations of the solution of the pde, such as the quasiinterpolant, the elliptic projection, and the $L^{2}$ projection. Here, the use of the semidiscrete spectral approximation itself for this purpose simplifies the analysis; however many technical difficulties remain and they are resolved in the course of the proof of Proposition \ref{propo43}.

In section \ref{sec5} we revert to the general $s$-stage temporal discretization scheme and, under the hypotheses that the solution of (\ref{dd12}) is in $H^{\mu}$ for $t\in [0,T]$ for $\mu$ sufficiently large, and that the local temporal error is of $O(k^{\alpha+1})$ in $L^{2}$, we prove that there exists a constant $C$ such that if $kN\leq C$, the fully discrete scheme has a unique solution whose maximum $L^{2}$ error over $[0,T]$ has a bound of $O(k^{\alpha}+N^{1-\mu})$. Therefore, the RK scheme with $s=3$ stages, whose local temporal error was analyzed in section \ref{sec44}, leads to a fully discrete method with an $L^{2}$ error bound of $O(k^{4}+N^{1-\mu})$. In a remark at the end of section \ref{sec5} we discuss the convergence of a simple iterative scheme approximating the nonlinear system of equations that must be solved at each IMR stage of the RK time-stepping scheme.

In a final section \ref{sec6} we summarize the results of the paper and indicate how they may be extended to the case of the generalized Benjamin equation, solved numerically with the present scheme in \cite{DDM2019}.

As was already mentioned we will denote by $H^{\mu}$, for real $\mu\geq 0$, the $L^{2}$-based Sobolev space of order $\mu$ consisting of periodic functions on $(-\pi,\pi)$. For $g\in H^{\mu}$ its norm is given by
\begin{eqnarray*}
||g||_{{\mu}}=\left(\sum_{k\in\mathbb{Z}}(1+k^{2})^{\mu}|\widehat{g}(k)|^{2}\right)^{1/2},
\end{eqnarray*}
 where $\widehat{g}(k)=\frac{1}{2\pi}\int_{-\pi}^{\pi}e^{-ikx}g(x)dx$ is the $k$th Fourier coefficient of $g$. For $1\leq p\leq\infty$ we denote by $W_{p}^{\mu}=W_{p}^{\mu}(-\pi,\pi)$ the real Sobolev space of periodic functions on $(-\pi,\pi)$ and denote its norm by $||\cdot||_{\mu,p}$, while $|\cdot|_{\infty}$ will stand for the norm of $L^{\infty}(-\pi,\pi)$. Finally, the inner product in $L^{2}=L^{2}(-\pi,\pi)$ will be defined by
$
(u,v)=\int_{-\pi}^{\pi}u(x)\overline{v(x)}dx,
$
and $||\cdot||$ will denote the induced $L^{2}$ norm.


\section{Semidiscretization and preliminaries}
\label{sec2}
Let $N\geq 1$ be an integer and consider the finite-dimensional space $S_{N}$ defined by
\begin{eqnarray*}
S_{N}=span\{e^{ikx}, \, k \,{\rm integer},\,-N\leq k\leq N\}.
\end{eqnarray*}
Let  $P_{N}$ denote the $L^{2}-$projection operator onto $S_{N}$ defined for $v\in L^{2}$ by
$$P_{N}v=\sum_{|k|\leq N}\widehat{v}_{k}e^{ikx},$$ where $\widehat{v}_{k}=\widehat{v}(k)$ is the $k-$th Fourier coefficient of $v$. We note some well-known properties of $P_{N}$ that will be used throughout the paper. It is obvious that $P_{N}$ commutes with the differentiation operator $\partial_{x}$. Moreover, cf. \cite{Me}, given integers $0\leq j\leq\mu$, there exists a constant $C$ independent of $N$ such that for any $v\in H^{\mu}, \mu\geq 1$,
\begin{eqnarray}
||v-P_{N}v||_{j}&\leq &CN^{j-\mu}||v||_{\mu},\label{dd21}\\ 
|v-P_{N}v|_{\infty}& \leq & CN^{1/2-\mu}||v||_{\mu}.\label{dd22}
\end{eqnarray}
In addition, the following inverse inequalities hold on $S_{N}$. Given $0\leq j\leq \mu$, there exists a constant $C_{0}$ independent of $N$, such that for all $\psi\in S_{N}$
\begin{eqnarray}
||\psi||_{\mu}\leq C_{0}N^{\mu-j}||\psi||_{j},\quad 
||\psi||_{\mu,\infty}\leq C_{0}N^{1/2+\mu-j}||\psi||_{j}.\label{dd23}
\end{eqnarray}

The semidiscrete Fourier-Galerkin approximation to the solution of  (\ref{dd12}) is a real-valued map $u^{N}:[0,T]\rightarrow S_{N}$ such that, for all $\chi\in S_{N}$,
\begin{eqnarray}
&&(u_{t}^{N}+u^{N} u_{x}^{N}+u^{N}_{xxx},\chi)=0,\quad 0\leq t\leq T,\label{dd24}\\
&&u^{N}(x,0)=P_{N}u_{0}(x).\nonumber
\end{eqnarray}
It is straightforward to see that while the solution of (\ref{dd24}) exists, it satisfies
\begin{eqnarray*}
\frac{d}{dt}\int_{-\pi}^{\pi} u^{N} dx&=&0,\\
\frac{d}{dt}\int_{-\pi}^{\pi} (u^{N})^{2} dx&=&0,\\
\frac{d}{dt}\int_{-\pi}^{\pi} \left((u_{x}^{N})^{2}-\frac{1}{3} (u^{N})^{3}\right)dx&=&0.
\end{eqnarray*}
In particular, while $u^{N}$ exists, we have
\begin{eqnarray}
||u^{N}(t)||=||u^{N}(0)||,\label{dd25}
\end{eqnarray}
from which, from standard ode theory, we see that $u^{N}(t)$ exists uniquely for all $t>0$ and, in particular, satisfies (\ref{dd25}) and the other conservation laws for $0\leq t\leq T$.
\section{Convergence of the semidiscretization}
\label{sec3}
\begin{theorem}
\label{th31}
Let $u^{N}$ be the solution of (\ref{dd24}) and suppose that $u$, the solution of (\ref{dd12}), belongs to $H^{\mu}, \mu\geq 2$ for $t\in [0,T]$. Then for some constant $C$ independent of $N$ it holds that
\begin{eqnarray}
\max_{0\leq t\leq T}||u^{N}-u||\leq \frac{C}{N^{\mu-1}}.\label{dd31}
\end{eqnarray}
\end{theorem}
\begin{proof}
We write $u^{N}-u=\theta+\rho$, where $\theta=u^{N}-P_{N}u\in S_{N}$, and $\rho=P_{N}u-u$. We then have for $\chi\in S_{N}, 0\leq t\leq T$, in view of (\ref{dd24}), (\ref{dd12}),
\begin{eqnarray*}
(\theta_{t},\chi)+(\theta_{xxx},\chi)=(u_{t}^{N}+u_{xxx}^{N},\chi)-(P_{N}(u_{t}+u_{xxx}),\chi)
=(uu_{x}-u^{N}u_{x}^{N},\chi).
\end{eqnarray*}
Since 
\begin{eqnarray*}
uu_{x}-u^{N}u_{x}^{N}&=&uu_{x}-(u+\theta+\rho)(u_{x}+\theta_{x}+\rho_{x})\\
&=&-u\theta_{x}-u\rho_{x}-u_{x}\theta-\theta\theta_{x}-\rho_{x}\theta-u_{x}\rho-\rho\theta_{x}-\rho\rho_{x},
\end{eqnarray*}
we have for $\chi\in S_{N}$
\begin{eqnarray*}
(\theta_{t},\chi)+(\theta_{xxx},\chi)&=&-\left((u\theta_{x},\chi)+(u\rho_{x},\chi)+(u_{x}\theta,\chi)+(\theta\theta_{x},\chi)\right.\\
&&+\left.(\rho_{x}\theta,\chi)+(u_{x}\rho,\chi)+(\rho\theta_{x},\chi)+(\rho\rho_{x},\chi)\right).
\end{eqnarray*}
Putting $\chi=\theta$ in the above and using integration by parts and periodicity we obtain for $0\leq t\leq T$
\begin{eqnarray}
\frac{1}{2}\frac{d}{dt}||\theta||^{2}=-\left(\frac{1}{2}(u_{x},\theta^{2})+(u\rho_{x},\theta)+\frac{1}{2}(\rho_{x},\theta^{2})+(u_{x}\rho,\theta)+(\rho\rho_{x},\theta)\right).\label{dd32}
\end{eqnarray}
We estimate now the various inner products in the right-hand side of the above, taking into account that $u\in H^{\mu}, \mu\geq 2$. We first have
\begin{eqnarray}
|(u_{x},\theta^{2})|\leq |u_{x}|_{\infty}||\theta||^{2}\leq C||\theta||^{2}.\label{dd33a}
\end{eqnarray}
(Here and in the sequel $C$ will denote a generic constant independent of the discretization parameters.)
By (\ref{dd21})
\begin{eqnarray}
|(u\rho_{x},\theta)|\leq |u|_{\infty}||\rho_{x}||||\theta||\leq CN^{1-\mu}||\theta||.\label{dd33b}
\end{eqnarray}
Using the inequality $|v_{x}|_{\infty}\leq C||v_{x}||^{1/2}||v_{xx}||^{1/2}$, valid in $H^{2}$, we see, in view of (\ref{dd21}), since $\mu\geq 2$
\begin{eqnarray}
|(\rho_{x},\theta^{2})|\leq |\rho_{x}|_{\infty}||\theta||^{2}\leq CN^{\frac{3}{2}-\mu}||\theta||^{2}\leq C||\theta||^{2}.\label{dd33c}
\end{eqnarray}
Also, by (\ref{dd21})
\begin{eqnarray}
|(u_{x}\rho,\theta)|\leq |u_{x}|_{\infty}||\rho||||\theta||\leq CN^{-\mu}||\theta||.\label{dd33d}
\end{eqnarray}
And, as above
\begin{eqnarray}
|(\rho\rho_{x},\theta)|\leq |\rho|_{\infty}||\rho_{x}||||\theta||\leq CN^{\frac{3}{2}-2\mu}||\theta||\leq CN^{-\mu}||\theta||.\label{dd33e}
\end{eqnarray}
We conclude by (\ref{dd32})-(\ref{dd33e}) that
\begin{eqnarray*}
\frac{d}{dt}||\theta||^{2}\leq C\left(N^{2(1-\mu)}+||\theta||^{2}\right),\;\;0\leq t\leq T,
\end{eqnarray*}
from which, by Gronwall's lemma, since $\theta(0)=0$, we get
\begin{eqnarray*}
\max_{0\leq t\leq T}||\theta||\leq CN^{1-\mu},
\end{eqnarray*}
and (\ref{dd31}) follows, in view of (\ref{dd21}).
\end{proof}

For the purposes of estimating the error of the temporal discretization of (\ref{dd24}), we note the following boundedness result for the semidiscrete approximation $u^{N}$, which is a consequence of the error estimate (\ref{dd31}).
\begin{proposition}
\label{propo32}
Let $u^{N}$ be the solution of (\ref{dd24}) and suppose that the solution $u$ of (\ref{dd12}) belongs to $H^{\mu}$ for $t\in [0,T]$. Then, given nonnegative integers $j$ and $l$, and provided $\mu\geq \max \{2,3j+l+1\}$, there exists a constant $C$ independent of $N$ such that
\begin{eqnarray}
\max_{0\leq t\leq T}||\partial_{t}^{j}u^{N}||_{l}\leq C.\label{dd34}
\end{eqnarray}
\begin{proof}
Using (\ref{dd21}), (\ref{dd23}), and (\ref{dd31}), provided $\mu\geq 2$, we have
\begin{eqnarray*}
||u^{N}||_{l}&\leq &||u-P_{N}u||_{l}+||P_{N}u-u^{N}||_{l}+||u||_{l}\\
&\leq & CN^{l-\mu}||u||_{\mu}+CN^{l}\left(||P_{N}u-u||+||u-u^{N}||\right)+||u||_{l}\\
&\leq &CN^{l-\mu}||u||_{\mu}+CN^{l+1-\mu}||u||_{\mu}+||u||_{l}.
\end{eqnarray*}
Therefore, if $\mu\geq \max \{2,l+1\}$ it holds that
\begin{eqnarray}
\max_{0\leq t\leq T}||u^{N}||_{l}\leq C.\label{dd35}
\end{eqnarray}
Since
\begin{eqnarray}
\partial_{t}u^{N}=-u_{xxx}^{N}-P_{N}(u^{N}u_{x}^{N}),\label{dd36}
\end{eqnarray}
we have
\begin{eqnarray*}
||\partial_{t}u^{N}||_{l}\leq ||u^{N}||_{l+3}+||P_{N}(u^{N}u_{x}^{N})||_{l}\leq ||u^{N}||_{l+3}+||u^{N}u_{x}^{N}||_{l}.
\end{eqnarray*}
From Sobolev's theorem and the fact that $H^{l}$ is an algebra for $l\geq 1$ we conclude from the above that
\begin{eqnarray*}
||\partial_{t}u^{N}||_{l}\leq ||u^{N}||_{l+3}+C||u^{N}||_{l+1}^{2}.
\end{eqnarray*}
Therefore, in view of (\ref{dd35}) and if $\mu\geq l+4$ we have
\begin{eqnarray}
||\partial_{t}u^{N}||_{l}\leq C.\label{dd37}
\end{eqnarray}
Finally, differentiating (\ref{dd36}) $j-1$ times with respect to $t$ and using repeatedly (\ref{dd35}) and (\ref{dd37}), we obtain (\ref{dd34}).
\end{proof}
\end{proposition}

\section{Full discretization by a Runge-Kutta method of composition type}
\label{sec4}
As mentioned in the Introduction, we will discretize in time the initial-value problem for the system of ordinary differential equations (ode ivp) represented by (\ref{dd24}), using an implicit $s$-stage RK-composition scheme based on the Implicit Midpoint Rule (IMR). In this section, after reviewing briefly the IMR, we will present the time-stepping RK method to be analyzed, study the existence of solutions of the resulting fully discrete scheme, its $L^{2}-$conservation property, present a preliminary uniqueness of solutions result and, for $s=3$, prove an $L^{2}$ estimate of its local temporal error.

\subsection{Fully discrete scheme with Implicit Midpoint Rule time stepping}
\label{sec41}
A simple time-stepping method that may be used to discretize the ode ivp (\ref{dd24}) in $t$ is the {\it Implicit Midpoint Rule} (IMR), which, in the case of the autonomous ode system $\dot{y}=\phi(y)$, is the single-step scheme
$$y^{n+1}-y^{n}=k\phi(y^{n+1/2}),$$ where $k$ here and the sequel will denote the (uniform) time step, $y^{n}$ is the approximation of $y(t^{n}), t^{n}=nk$, and $y^{n+1/2}=\frac{1}{2}(y^{n+1}+y^{n})$. In the case of the ivp (\ref{dd24}), assuming that $T=Mk$ where $M$ is an integer, the scheme is the following: We seek $U^{n}\in S_{N}$ for $n=0,\ldots,M$, satisfying for each $\chi\in S_{N}$

\begin{eqnarray}
(U^{n+1}-U^{n},\chi)&=&k\left(-(U^{n+1/2})_{xxx}-f(U^{n+1/2})_{x},\chi\right),\label{41}\\
U^{0}&=&P_{N}u_{0},\nonumber
\end{eqnarray}
where here and in the sequel we put $f(v)=v^{2}/2$ and
$U^{n+1/2}=\frac{U^{n+1}+U^{n}}{2}$. For the Fourier coefficients $\widehat{U^{n}}(j), -N\leq j\leq N,$ of $U^{n}$ we may write

\begin{eqnarray*}
&&\frac{\widehat{U^{n+1}}(j)-\widehat{U^{n}}(j)}{k}=i\left(j^{3}\widehat{U^{n+1/2}}(j)-j\widehat{f(U^{n+1/2})}(j)\right),\nonumber\\
&&\widehat{U^{0}}(j)=\widehat{u_{0}}(j),\quad -N\leq j\leq N,\nonumber
\end{eqnarray*}
where $\widehat{f(U^{n+1/2})}(j)$ denotes the $j-$th Fourier coefficient of ${f(U^{n+1/2})}$.

It is easy to see, by writing for each $n$ the equations (\ref{41}) in fixed-point form and applying a variant of Brouwer's fixed point theorem, that, given $U^{n}\in S_{N}$, there exists a solution $U^{n+1}\in S_{N}$ of the nonlinear system of equations represented by (\ref{41}). Putting $\chi=U^{n}+U^{n+1}$ and using periodicity one may also obtain that the method is $L^{2}-$conservative, i.~e. that
\begin{eqnarray}
||U^{n}||=||U^{0}||,\quad 0\leq n\leq M.\label{42}
\end{eqnarray}
By comparing $U^{n}$ with $u^{N}(t^{n})$ where $u^{N}$ is the solution of (\ref{dd24}), using (\ref{dd31}) and (\ref{dd34}), one may derive in a straightforward way the following error estimate for $U^{n}$.
\begin{proposition}
\label{propo41}
Suppose that $u$, the solution of (\ref{dd12}), belongs to $H^{\mu}$ for $t\in [0,T]$, where $\mu\geq 10$. Then, there exists a constant $\alpha>0$, such that if $k\leq \frac{\alpha}{N}$, there exists a unique solution $\{U^{n}\}_{n=0}^{M}$ of (\ref{41}) satisfying
\begin{eqnarray}
\max_{0\leq n\leq M}||U^{n}-u(t^{n})||\leq C(k^{2}+N^{1-\mu}),\label{43}
\end{eqnarray}
where $C$ is a constant independent of $N$ and $k$.
\end{proposition}
As the error analysis of the IMR fully discrete scheme (\ref{41}) may be viewed as a special case of the convergence proof for the general fully discrete scheme to be considered in the sequel, we will not present the proof of Proposition \ref{propo41} here.
\subsection{A Runge-Kutta-composition method}
\label{sec42}
We will consider a RK method with $s$ stages for the autonomous ode system $\dot{y}=\phi(y)$, whose Butcher tableau is of the form

\begin{equation}\label{44}
\begin{tabular}{c | c}
& $a_{ij}$\\ \hline
 & $b_{i}$
\end{tabular}=
\begin{tabular}{c | ccccc}
& $b_{1}/2$ & & &&\\
&  $b_{1}$ & $b_{2}/2$ &&&\\
&$b_{1}$&$b_{2}$&$\ddots$&&\\
&$\vdots$&$\vdots$ &&$\ddots$&\\
&$b_{1}$&$b_{2}$&$\cdots$&$\cdots$&$b_{s}/2$\\ \hline
 & $b_{1}$&$b_{2}$&$\cdots$&$\cdots$&$b_{s}$
\end{tabular},
\end{equation}
where the $b_{i}$ are nonzero real numbers. As has been pointed out in \cite{S-SA} all symplectic (canonical) RK schemes, i.~e. those satisfying $b_{i}a_{ij}+b_{j}a_{ji}-b_{i}b_{j}=0, i\leq i,j\leq s$, with lower triangular matrix $a_{ij}$ (i.~e. all diagonally implicit symplectic schemes) are of the form (\ref{44}).

It is well known, cf. e~g. \cite{Yoshida1990}, \cite{FrutosS1992}, \cite{SanzSC1994}, \cite{HairerLW2004}, and their references, that the RK scheme corresponding to the tableau (\ref{44}) is of {\it composition} type, since it may be constructed as the composition of $s$ steps of the IMR with stepsizes $b_{1}k, b_{2}k,\ldots,b_{s}k$, i.~e. in the case of $\dot{y}=\phi(y)$ it is equivalent to the scheme
\begin{eqnarray}
y^{n,1}&=&y^{n}+b_{1}k\phi\left(\frac{y^{n}+y^{n,1}}{2}\right),\label{45}\\
y^{n,j}&=&y^{n}+b_{j}k\phi\left(\frac{y^{n,j-1}+y^{n,j}}{2}\right),\;2\leq j\leq s,\nonumber\\
y^{n+1}&=&y^{n,s}.\nonumber
\end{eqnarray}
For example, a method mentioned in the above references and used in \cite{FrutosS1992} for the temporal discretization of the KdV equation corresponds to $s=3$ and
\begin{eqnarray}
&&b_{1}=(2+2^{1/3}+2^{-1/3})/3=\frac{1}{2-2^{1/3}}\cong 1.351,\nonumber\\
&& b_{2}=1-2b_{1}\cong -1.702,\quad b_{3}=b_{1},\label{46}
\end{eqnarray}
has order of accuracy $p=4$ and is symmetric since $b_{3}=b_{1}$. This scheme may be generalized using Yoshida's approach, \cite{Yoshida1990}, that yields recursively symplectic symmetric methods (in our case taking the IMR as base scheme) as follows. Let $\psi_{k}^{[2]}$ be the mapping that effects the step $n\mapsto n+1$ of the IMR with stepsize $k$. Then the method with $s=3$ may be viewed, in the notation of \cite{HairerLW2004}, as the composition
$$
\psi_{k}^{[4]}=\psi_{b_{3}k}^{[2]}\circ\psi_{b_{2}k}^{[2]}\circ\psi_{b_{1}k}^{[2]}.
$$ From $\psi_{k}^{[4]}$ one gets the sixth-order accurate symmetric method
with $s=3^{2}$
$$
\psi_{k}^{[6]}=\psi_{\gamma_{3}k}^{[4]}\circ\psi_{\gamma_{2}k}^{[4]}\circ\psi_{\gamma_{1}k}^{[4]},
$$ where $\gamma_{1}=\gamma_{3}=\frac{1}{2-2^{1/5}}, \gamma_{2}=1-2\gamma_{1}$. In general, given the method $\psi_{k}^{[2r]}$ with $s=3^{r-1}$ stages and order of accuracy $2r$, one may construct a symmetric scheme with $s=3^{r}$ stages of order of accuracy $2r+2$ by the formula
$$
\psi_{k}^{[2r+2]}=\psi_{\delta_{3,r}k}^{[2r]}\circ\psi_{\delta_{2,r}k}^{[2r]}\circ\psi_{\delta_{1,r}k}^{[2r]},
$$ where $\delta_{1,r}=\delta_{3,r}=\frac{1}{2-2^{\frac{1}{2r+1}}}, \delta_{2,r}=1-2\delta_{1,r}$.

Some properties of the resulting family of $s-$stage methods are summarized below.
\begin{itemize}
\item[(i)] The number of stages is $s=3^{p-1}$ and the order of accuracy of the method is $2p$, cf. \cite{Yoshida1990}.
\item[(ii)] $\displaystyle\sum_{i=1}^{s} b_{i}=1,\quad \displaystyle\sum_{i=1}^{s} b_{i}^{j}=0$, with $s=3^{p-1}$ and $j=3,5,\ldots,2p-1$, cf. \cite{Yoshida1990}.
\item[(iii)] The methods are symplectic and symmetric. Thus, when applied to an ode system with a Hamiltonian structure, they will preserve important properties of the system and behave well in long-time computations, \cite{FrutosS1992,SanzSC1994,HairerLW2004,Cano}.
\item[(iv)] Since some of the $b_{i}$ are negative, cf. e.~g. (\ref{46}) and property (ii) above, these schemes are not A-stable. They are however absolutely stable in a strip of finite width in $Re(z)\leq 0$ including the imaginary axis, and therefore it is expected that a stepsize restriction will be needed for stability in the case of dissipative problems.
\item[(v)] The implementation of the schemes is straightforward as it requires solving $s$ nonlinear systems of the size of the ode system, as is evident from e.~g. (\ref{45}).
\end{itemize}
\subsection{The fully discrete scheme. $L^{2}-$conservation, existence and uniqueness of solutions}
\label{sec43}

The high accuracy, straightforward manner of implementation, and the good stability properties (in the case of conservative, stiff ode systems) of the family of RK composition methods given by (\ref{44}) (equivalently by (\ref{45})), make them a good choice as time-stepping schemes for the semidiscrete ivp (\ref{dd24}). 

As mentioned already, the fully discrete scheme to be fully analyzed in the sequel is obtained by applying the $s-$stage RK compositon method (\ref{44}) or (\ref{45}) to the semidiscrete problem (\ref{dd24}) when $s=3$ and the coefficients $b_{i}$ are given by (\ref{46}). However, with the exception of the estimation of the local temporal error in section \ref{sec44}, the rest of the proof of convergence holds for the general $s-$stage scheme and therefore we will treat the general case and specialize $s=3$ when needed. 
For simplifying notation we let $F:S_{N}\rightarrow S_{N}$ be the nonlinear map defined for $v\in S_{N}$ by the equation
\begin{eqnarray*}
(F(v),\chi)=(-v_{xxx}-P_{N}f(v)_{x},\chi), \quad \forall \chi\in S_{N},\label{47}
\end{eqnarray*}
where $f(v)=v^{2}/2$, or equivalently, by
\begin{eqnarray*}
F(v)=-v_{xxx}-P_{N}f(v)_{x}.\label{48}
\end{eqnarray*}
Note that, by periodicity,
\begin{eqnarray}
(F(v),v)=0,\quad \forall v\in S_{N},\label{49}
\end{eqnarray}
and that the semidiscrete ivp (\ref{dd24}) may be written as
\begin{eqnarray}
u_{t}^{N}&=&F(u^{N}),\quad 0\leq t\leq T,\nonumber\\
u^{N}(0)&=&P_{N}u_{0}.\label{410}
\end{eqnarray}
Using the notation introduced for the temporal discretization in section \ref{sec41} we write the RK scheme (\ref{44}) applied to (\ref{410}) as follows. For $0\leq n\leq M$ we seek $U^{n}\in S_{N}$, approximating $u^{N}(t^{n})$, and $U^{n,i}\in S_{N}, 1\leq i\leq s$, such that for $0\leq n\leq M-1$
\begin{eqnarray}
U^{n,i}&=&U^{n}+\frac{b_{i}k}{2}F(U^{n,i})+k \sum_{j=1}^{i-1}b_{j}F(U^{n,j}),\quad 1\leq i\leq s,\nonumber\\
U^{n+1}&=&U^{n}+k\sum_{i=1}^{s}b_{i}F(U^{n,i}),\label{411}
\end{eqnarray}
and
$$U^{0}=P_{N}u_{0}.$$ By eliminating recursively the intermediate nonlinear terms and defining $\mu_{ij}=2(-1)^{i+j+1}$, $1\leq j<i\leq s$, it is easy to check that the scheme (\ref{411}) may be equivalently stated for $0\leq n\leq M-1$ as
\begin{eqnarray}
U^{n,i}&=&(-1)^{i+1}U^{n}+\frac{b_{i}k}{2}F(U^{n,i})+\sum_{j=1}^{i-1}\mu_{ij}U^{n,j},\quad 1\leq i\leq s,\nonumber\\
U^{n+1}&=&(-1)^{s}U^{n}+2\sum_{j=1}^{s}(-1)^{s-j}U^{n,j},\label{412}
\end{eqnarray}
and
$$U^{0}=P_{N}u_{0}.$$ As already mentioned, the scheme (\ref{411}) is also equivalent to the following IMR-type formulation (cf. (\ref{45})) in which, given $U^{n}\in S_{N}, 0\leq n\leq M-1$, $Y^{n,i}\in S_{N}, i\leq i\leq s$, and $U^{n+1}\in S_{N}$ are computed by the formulas
\begin{eqnarray}
Y^{n,1}&=&U^{n} +{kb_{1}}{F}\left(\frac{Y^{n,1}+U^{n}}{2}\right),\nonumber\\
Y^{n,i}&=&Y^{n,i-1} +{kb_{i}}{F}\left(\frac{Y^{n,i}+Y^{n,i-1}}{2}\right),\quad 2\leq i\leq s,\nonumber\\
U^{n+1}&=&Y^{n,s}.\label{413}
\end{eqnarray} 
and
$$U^{0}=P_{N}u_{0}.$$
Note that the intermediate approximations $Y^{n,i}$ of (\ref{413}) are related to the $U^{n,i}$ of (\ref{411}) or (\ref{412}) by the formulas
\begin{eqnarray*}
Y^{n,i}=2U^{n,i}-Y^{n,i-1},\; 2\leq i\leq s,\quad Y^{n,1}=2U^{n,1}-U^{n}.\label{414}
\end{eqnarray*}
Any one of the formulations (\ref{411})-(\ref{413}) may be used to study the properties and the convergence of the fully discrete scheme. We will mainly use (\ref{413}) which brings out the fact that the scheme is a $s-$stage composition method with IMR as its base scheme.

As is expected by the symplecticity of the RK method (\ref{44}), (\ref{45}), the fully discrete schemes (\ref{411})-(\ref{413}) are $L^{2}-$conservative. Taking, for example, (\ref{413}) and supposing that given $U^{n}\in S_{N}$ it has a solution $Y^{n,i}\in S_{N}, 1\leq i\leq s$, then, if $i\geq 2$
\begin{eqnarray*}
(Y^{n,i}-Y^{n,i-1},Y^{n,i}+Y^{n,i-1})=2kb_{i}\left({F}\left(\frac{Y^{n,i}+Y^{n,i-1}}{2}\right),\frac{Y^{n,i}+Y^{n,i-1}}{2}\right),
\end{eqnarray*}
which, in view of (\ref{49}) yields $||Y^{n,i-1}||=||Y^{n,i}||$. The same argument works for $i=1$ if we put $Y^{n,0}=U^{n}$ and yields $||Y^{n,1}||=||U^{n}||$. Therefore $||U^{n+1}||=||Y^{n,i}||=||U^{n}||, 1\leq i\leq s$, and overall
$$||U^{n}||=||U^{0}||,\; 0\leq n\leq M,$$ provided $U^{n}, 1\leq n\leq M$ exist. The existence of solutions may be proved by a variant of Brouwer's fixed-point theorem. We use again (\ref{413}).
\begin{proposition}
\label{propo42}
Given $U^{n}\in S_{N}$, there are $Y^{n,i}\in S_{N}, 1\leq i\leq s$, and $U^{n+1}$ in $S_{N}$ satisfying (\ref{413}).
\end{proposition}

\begin{proof} Putting $Z=\frac{Y^{n,1}+U^{n}}{2}$, we write the first equation in (\ref{413}) in the form $Z-U^{n}=\frac{kb_{1}}{2}F(Z)$. Hence, if we define $G:S_{N}\rightarrow S_{N}$ for $v\in S_{N}$ as $G(v)=v-U^{n}-\frac{kb_{1}}{2}F(v)$, for $\chi\in S_{N}$ we have
$(G(v),\chi)=(v-U^{n},\chi)-\frac{kb_{1}}{2}(F(v),\chi)$. Taking $\chi=v$ we get, in view of (\ref{49}), $(G(v),v)=||v||^{2}-(U^{n},v)\geq ||v||\left(||v||-||U^{n}||\right)$. Therefore, if $||v||=||U^{n}||$, then $(G(v),v)\geq 0$. By the definition of $F$ and the inverse inequalities (\ref{dd23}) it follows that $F$, and hence $G$, is continuous on $S_{N}$. By a well-known variant of Brouwer's fixed-point theorem (see e.~g. Lemma 3.1 of \cite{BonaDKM1995}), there exists $Z\in S_{N}$ with $||Z||=||U^{n}||$, such that $G(Z)=0$, i.~e. $Z-U^{n}=\frac{kb_{1}}{2}F(Z)$, and the existence of $Y^{n,1}$ follows. (For $Y^{n,1}$ we know {\it a priori} that $||Y^{n,1}||=||U^{n}||$.) In an analogous way we may prove recursively the existence of $Y^{n,i}, 2\leq i\leq s$, satisfying (\ref{413}). 
\end{proof}

The uniqueness of solutions of the nonlinear systems represented by the nonlinear equations in (\ref{413}) will be shown in the course of the proof of convergence of the fully discrete scheme in section \ref{sec5}. The following lemma establishes uniqueness under a condition that will be verified in section \ref{sec5}.
\begin{lemma}
\label{lemma41}
Suppose that $U^{n}$ and $Y^{n,i}, 1\leq i\leq s$, are solutions of (\ref{413}) satisfying $|U^{n}|_{\infty}\leq R, |Y^{n,i}|_{\infty}\leq R, 1\leq i\leq s$, for some constant $R$. Then the $Y^{n,i}, 1\leq i\leq s$, are unique, provided that
\begin{eqnarray*}
\frac{k}{2}\max_{1\leq i\leq s}|b_{i}|C_{0}NR<1,
\end{eqnarray*}
where $C_{0}$ is the constant in the inverse properties (\ref{dd23}). 
\end{lemma}
\begin{proof} We prove the uniqueness of $Y^{n,1}$; that of $Y^{n,i}, i\geq 2$, follows by a similar argument. Suppose $Z_{1}, Z_{2}\in S_{N}$ are two solutions of the first equation in (\ref{413}). (Note that $||Z_{1}||=||Z_{2}||=||U^{n}||$.) Then $Z_{1}-Z_{2}=kb_{1}\left(F\left(\frac{Z_{1}+U^{n}}{2}\right)-F\left(\frac{Z_{2}+U^{n}}{2}\right)\right)$. Taking the inner product of both sides of this equation with $Z_{1}-Z_{2}$ and using periodicity, the Cauchy-Schwarz inequality and (\ref{dd23}), we have if $Z_{1}\neq Z_{2}$ that
\begin{eqnarray*}
||Z_{1}-Z_{2}||\leq k|b_{1}|C_{0}N\left|\left|f\left(\frac{Z_{1}+U^{n}}{2}\right)-f\left(\frac{Z_{2}+U^{n}}{2}\right)\right|\right|.
\end{eqnarray*}
By the definition of $f$ and the hypothesis of the lemma we see that
\begin{eqnarray*}
\left|\left|f\left(\frac{Z_{1}+U^{n}}{2}\right)-f\left(\frac{Z_{2}+U^{n}}{2}\right)\right|\right|&\leq &\frac{||Z_{1}-Z_{2}||}{2}\frac{1}{4}|Z_{1}+U^{n}+Z_{2}+U^{n}|_{\infty}\\
&\leq & \frac{1}{2}||Z_{1}-Z_{2}||R.
\end{eqnarray*}
These two inequalities imply that $1\leq \frac{1}{2}k|b_{1}|C_{0}NR$, which contradicts the other hypothesis of the lemma. Therefore $Z_{1}=Z_{2}$. 
\end{proof}
\subsection{Local temporal error of the fully discrete scheme for $s=3$}
\label{sec44}
In this section we suppose that $s=3$ and that the coefficients $b_{i}$ are given by (\ref{46}).
The local temporal error of the resulting scheme (\ref{413}) is defined in terms of the semidiscrete approximation $u^{N}$. For this purpose we let for $0\leq n\leq M$, $V^{n}=u^{N}(t^{n})$, and $V^{n,i}\in S_{N}$ for $0\leq i\leq 3, 0\leq n\leq M-1$, be given by
\begin{eqnarray}
V^{n,0}&=&V^{n},\nonumber\\
V^{n,i}&=&V^{n,i-1}+kb_{i}F\left(\frac{V^{n,i}+V^{n,i-1}}{2}\right),\; 1\leq i\leq 3.\label{415}
\end{eqnarray}
The local temporal error $\theta^{n}\in S_{N}$, $0\leq n\leq M-1$, is then
\begin{eqnarray}
\theta^{n}=V^{n+1}-V^{n,3}\equiv u^{N}(t^{n+1})-V^{n,3}.\label{416}
\end{eqnarray}
Obviously, cf. section \ref{sec43}, the $V^{n,i}$ exist and satisfy the $L^{2}$-conservation laws
\begin{eqnarray*}
||V^{n,i}||=||V^{n}||=||u^{N}(t^{n})||=||u^{N}(0)||.\label{417}
\end{eqnarray*}
The consistency of the scheme is established in the following
\begin{proposition}
\label{propo43}
Let $V^{n,i}$ and $\theta^{n}$ be defined by (\ref{415}) and (\ref{416}) and suppose the $b_{i}$ are given by (\ref{46}). Let $u$, the solution of (\ref{dd12}), belong to $H^{\mu}$ for $0\leq t\leq T$ and let $\mu$ be sufficiently large. Suppose there exists a constant $C_{1}$ such that $kN\leq C_{1}$. Then, for $k$ sufficiently small, there exists a constant $C$ independent of $k$ and $N$ such that
\begin{eqnarray*}
\max_{0\leq n\leq M-1}||\theta^{n}||\leq Ck^{5}.
\end{eqnarray*}
\end{proposition}

\begin{proof}
The plan of the proof is to obtain
asymptotic expressions of the $V^{n,i}, i=1,2$, of the form
\begin{eqnarray}
V^{n,1}&=&u^{N}(\tau^{n,1})+A_{1}k^{3}+A_{2}k^{4}+e^{n,1},\label{418}\\
V^{n,2}&=&u^{N}(\tau^{n,2})+B_{1}k^{3}+B_{2}k^{4}+e^{n,2},\label{419}
\end{eqnarray}
where $\tau^{n,1}=t^{n}+kb_{1}, \tau^{n,2}=t^{n}+k(b_{1}+b_{2})$,
 $A_{i}, B_{i}, e^{n,i}\in S_{N}$ and $||e^{n,i}||\leq C k^{5}, i=1,2$; then we show that
\begin{eqnarray}
V^{n,3}=u^{N}(t^{n+1})+e^{n,3},\label{420}
\end{eqnarray}
where $||e^{n,3}||\leq C k^{5}$, implying that $||\theta^{n}||\leq C k^{5}$. These estimates will be uniformly valid in $n$. The coefficients $A_{i}$ and $B_{i}$ will be of $O(1)$. Here, and in the sequel, $C$ will denote generically constants independent of $k$ and $N$.

\bigskip
\noindent{(i)}
\underline{Asymptotic expansion of $V^{n,1}$. Determination of the coefficients $A_{1}, A_{2}$}.

From (\ref{415}) it follows for $i=1$
\begin{eqnarray}
V^{n,1}&=&u^{N}-kb_{1}\left(\partial_{x}^{3}\left(\frac{V^{n,1}+u^{N}}{2}\right)\right)\nonumber\\
&&-\frac{kb_{1}}{4}P_{N}\left((V^{n,1}+u^{N})(V^{n,1}+u^{N})_{x}\right).\label{421}
\end{eqnarray}
In (\ref{421}) and in the sequel we put $u^{N}=u^{N}(t^{n})$. Similarly we will suppress the argument $t^{n}$ from derivatives of $u^{N}$, i.~e. write $u_{x}^{N}=u_{x}^{N}(t^{n}), u_{t}^{N}=u_{t}^{N}(t^{n})$, etc. For the intermediate times $\tau^{n,i}$ we will write in full $u^{N}(\tau^{n,i})$, etc.

We now insert in (\ref{421}) the assumed expression (\ref{418}) for $V^{n,1}$ and obtain in the left-hand side by Taylor expansion
\begin{eqnarray}
V^{n,1}
&=&u^{N}+kb_{1}u_{t}^{N}+\frac{k^{2}b_{1}^{2}}{2}u_{tt}^{N}+\frac{k^{3}b_{1}^{3}}{6}u_{ttt}^{N}+\frac{k^{4}b_{1}^{4}}{24}\partial_{t}^{4}u^{N}\nonumber\\
&&+\rho_{1}+A_{1}k^{3}+A_{2}k^{4}+e^{n,1},\label{422}
\end{eqnarray}
where $\rho_{1}$, the Taylor remainder of the expansion of $u^{N}(\tau^{n,1})$ about $t^{n}$, satisfies
\begin{eqnarray}
||\rho_{1}||_{j}\leq C k^{5}\max_{t}||\partial_{t}^{5}u^{N}||_{j}.\label{423}
\end{eqnarray}
For the linear $\partial_{x}^{3}$-term in the right hand side of (\ref{421}) we have
\begin{eqnarray}
\partial_{x}^{3}\left(\frac{V^{n,1}+u^{N}}{2}\right)&=&\partial_{x}^{3}\left(u^{N}+\frac{kb_{1}}{2}u_{t}^{N}+\frac{k^{2}b_{1}^{2}}{4}u_{tt}^{N}+\frac{k^{3}b_{1}^{3}}{12}u_{ttt}^{N}+\rho_{2}\right)\nonumber\\
&&+\frac{1}{2}\left(k^{3}\partial_{x}^{3}A_{1}+k^{4}\partial_{x}^{3}A_{2}+\partial_{x}^{3}e^{n,1}\right),\label{424a}
\end{eqnarray}
where 
\begin{eqnarray}
||\rho_{2}||_{j}\leq C k^{4}\max_{t}||\partial_{t}^{4}u^{N}||_{j}.\label{424}
\end{eqnarray}
Finally for the last term in the right hand side of (\ref{421}) we see that
\begin{eqnarray}
-\frac{kb_{1}}{4}P_{N}\left((V^{n,1}+u^{N})(V^{n,1}+u^{N})_{x}\right)
&=&-\frac{kb_{1}}{4}P_{N}\left((u^{N}+u^{N}(\tau^{n,1}))(u^{N}+u^{N}(\tau^{n,1}))_{x}\right.\nonumber\\
&&\left.+k^{3}((u^{N}+u^{N}(\tau^{n,1}))A_{1})_{x}\right.\nonumber\\
&&+\left. k^{4}((u^{N}+u^{N}(\tau^{n,1}))A_{2})_{x}+k^{6}A_{1}A_{1,x}\right.\nonumber\\
&&\left.+k^{7}(A_{1}A_{2})_{x}+k^{8}A_{2}A_{2,x}+\mathcal{A}(e^{n,1})\right),\label{426}
\end{eqnarray}
where
\begin{eqnarray}
\mathcal{A}(e^{n,1})&=&((u^{N}+u^{N}(\tau^{n,1}))e^{n,1})_{x}+k^{3}(A_{1}e^{n,1})_{x}\nonumber\\
&&+k^{4}(A_{2}e^{n,1})_{x}+e^{n,1}e^{n,1}_{x}.\label{427}
\end{eqnarray}
We use now (\ref{422}) in the left-hand side and (\ref{424a}) and (\ref{426}) in the right-hand side of (\ref{421}) and obtain
\begin{eqnarray}
&&u^{N}+kb_{1}u_{t}^{N}+\frac{k^{2}b_{1}^{2}}{2}u_{tt}^{N}+\frac{k^{3}b_{1}^{3}}{6}u_{ttt}^{N}+\frac{k^{4}b_{1}^{4}}{24}\partial_{t}^{4}u^{N}+\rho_{1}+A_{1}k^{3}+A_{2}k^{4}+e^{n,1}\nonumber\\
&&=u^{N}-kb_{1}\partial_{x}^{3}\left(u^{N}+\frac{kb_{1}}{2}u_{t}^{N}+\frac{k^{2}b_{1}^{2}}{4}u_{tt}^{N}+\frac{k^{3}b_{1}^{3}}{12}u_{ttt}^{N}+\rho_{2}\right)\nonumber\\
&&-\frac{kb_{1}}{2}\left(k^{3}\partial_{x}^{3}A_{1}+k^{4}\partial_{x}^{3}A_{2}+\partial_{x}^{3}e^{n,1}\right)\nonumber\\
&&-\frac{kb_{1}}{4}P_{N}\left((u^{N}+u^{N}(\tau^{n,1}))(u^{N}+u^{N}(\tau^{n,1}))_{x}+k^{3}((u^{N}+u^{N}(\tau^{n,1}))A_{1})_{x}\right.\nonumber\\
&&+\left. k^{4}((u^{N}+u^{N}(\tau^{n,1}))A_{2})_{x}+k^{6}A_{1}A_{1,x}\right.\nonumber\\
&&+\left.k^{7}(A_{1}A_{2})_{x}+k^{8}A_{2}A_{2,x}+\mathcal{A}(e^{n,1})\right).\label{428}
\end{eqnarray}
We now equate the terms of equal powers of $k$ in the left- and right-hand side of (\ref{428}), after expanding the $u^{N}(\tau^{n,1})$ terms about $t^{n}$. The $O(1)$ terms are obviously identical. We also have:

{\it $O(k)$ terms}:
\begin{eqnarray*}
kb_{1}u_{t}^{N}&=&-kb_{1}\partial_{x}^{3}u^{N}-\frac{kb_{1}}{4}P_{N}\left(2u^{N}\cdot 2u_{x}^{N}\right)
=-kb_{1}\left(\partial_{x}^{3}u^{N}+P_{N}(u^{N}u_{x}^{N})\right),
\end{eqnarray*}
which is an identity in view of the definition of $u^{N}$, cf. (\ref{dd36}).

{\it $O(k^{2})$ terms}:
\begin{eqnarray*}
\frac{k^{2}b_{1}^{2}}{2}u_{tt}^{N}&=&-\frac{k^{2}b_{1}^{2}}{2}\partial_{x}^{3}u_{t}^{N}-\frac{kb_{1}}{4}P_{N}\left(2kb_{1}u^{N}u_{tx}^{N}+2kb_{1}u_{t}^{N}u_{x}^{N}\right)\\
&=&-\frac{k^{2}b_{1}^{2}}{2}\left(\partial_{x}^{3}u_{t}^{N}+P_{N}(u^{N}u_{x}^{N})_{t}\right),
\end{eqnarray*}
which is an identity. (Differentiate (\ref{dd36}) with respect to $t$.)

{\it $O(k^{3})$ terms}:
\begin{eqnarray}
\frac{k^{3}b_{1}^{3}}{6}\partial_{t}^{3}u^{N}+k^{3}A_{1}=-\frac{k^{3}b_{1}^{3}}{4}\partial_{x}^{3}u_{tt}^{N}-kb_{1}T_{2},\label{429}
\end{eqnarray}
where $T_{2}$ is the sum of the $O(k^{2})$ terms of the Taylor expansion of
$$P_{N}\left(\left(\frac{u^{N}+u^{N}(\tau^{n,1})}{2}\right)\left(\frac{u^{N}+u^{N}(\tau^{n,1})}{2}\right)_{x}\right),$$ about $t=t^{n}$ and is given by
\begin{eqnarray*}
T_{2}
&=&P_{N}\left(\frac{k^{2}b_{1}^{2}}{4}u^{N}u_{xtt}^{N}+\frac{k^{2}b_{1}^{2}}{4}u_{t}^{N}u_{xt}^{N}+\frac{k^{2}b_{1}^{2}}{4}u_{tt}^{N}u_{x}^{N}\right).
\end{eqnarray*}
From Leibniz's rule for differentiation of products we have
\begin{eqnarray*}
(vv_{x})_{tt}=vv_{xtt}+2v_{t}v_{xt}+v_{tt}v_{x}.
\end{eqnarray*}
Therefore
\begin{eqnarray}
T_{2}=\frac{k^{2}b_{1}^{2}}{4}P_{N}\left((u^{N}u_{x}^{N})_{tt}-u_{t}^{N}u_{xt}^{N}\right).\label{430}
\end{eqnarray}
Substituting (\ref{430}) into (\ref{429}) and using the fact 
 that $-\partial_{x}^{3}u_{tt}^{N}=u_{ttt}^{N}+P_{N}(u^{N}u_{x}^{N})_{tt}$, which follows by differentiation from (\ref{dd36}) we get
\begin{eqnarray*}
\frac{b_{1}^{3}}{6}\partial_{t}^{3}u^{N}+A_{1}&=&\frac{b_{1}^{3}}{4}u_{ttt}^{N}+\frac{b_{1}^{3}}{4}P_{N}(u^{N}u_{x}^{N})_{tt}
-\frac{b_{1}^{3}}{4}P_{N}\left((u^{N}u_{x}^{N})_{tt}\right)+\frac{b_{1}^{3}}{4}P_{N}\left(u_{t}^{N}u_{xt}^{N}\right),
\end{eqnarray*}
from which it follows that
\begin{eqnarray}
A_{1}=b_{1}^{3}\left(\frac{1}{12}u_{ttt}^{N}+\frac{1}{4}P_{N}\left(u_{t}^{N}u_{xt}^{N}\right)\right).\label{431}
\end{eqnarray}

{\it $O(k^{4})$ terms}:
\begin{eqnarray}
\frac{k^{4}b_{1}^{4}}{24}\partial_{t}^{4}u^{N}+k^{4}A_{2}&=&-\frac{k^{4}b_{1}^{4}}{12}\partial_{t}^{3}\partial_{x}^{3}u^{N}-\frac{k^{4}b_{1}}{2}\partial_{x}^{3}A_{1}-kb_{1}T_{3}-\frac{k^{4}b_{1}}{2}S_{0},\label{432}
\end{eqnarray}
where
\begin{eqnarray*}
T_{3}&=&P_{N}\left((\frac{u^{N}+u^{N}(\tau^{n,1})}{2})(\frac{u^{N}+u^{N}(\tau^{n,1})}{2})_{x}\right)\Big|_{O(k^{3})},\\
S_{0}&=&P_{N}\left((\frac{u^{N}+u^{N}(\tau^{n,1})}{2})A_{1})\right)_{x}\Big|_{O(1)}.
\end{eqnarray*}
Hence
\begin{eqnarray*}
T_{3}
&=&P_{N}\left(\frac{k^{3}b_{1}^{3}}{12}u^{N}u_{xttt}^{N}+\frac{k^{3}b_{1}^{3}}{8}u_{t}^{N}u_{xtt}^{N}+\frac{k^{3}b_{1}^{3}}{8}u_{tt}^{N}u_{xt}^{N}+\frac{k^{3}b_{1}^{3}}{12}u_{ttt}^{N}u_{x}^{N}\right).
\end{eqnarray*}
Leibniz's rule gives 
\begin{eqnarray*}
(vv_{x})_{ttt}=vv_{xttt}+3v_{t}v_{txt}+3v_{tt}v_{xt}+v_{ttt}v_{x}.
\end{eqnarray*}
Therefore
\begin{eqnarray}
-kb_{1}T_{3}=-\frac{k^{4}b_{1}^{4}}{12}P_{N}\left((u^{N}u_{x}^{N})_{ttt}-\frac{3}{2}\left(u_{t}^{N}u_{xtt}^{N}+u_{tt}^{N}u_{xt}^{N}\right)\right).\label{433}
\end{eqnarray}
In addition, in view of (\ref{431}) we get
\begin{eqnarray}
-\frac{k^{4}b_{1}}{2}S_{0}=-\frac{k^{4}b_{1}^{4}}{2}\left(\frac{1}{12}P_{N}\left(u^{N}u_{ttt}^{N}\right)_{x}+\frac{1}{4}P_{N}\left(u^{N}P_{N}\left(u_{t}^{N}u_{xt}^{N}\right)\right)_{x}\right).\label{434}
\end{eqnarray}
Therefore, by (\ref{431})-(\ref{434}), using Leibniz's formula for $(vv_{x})_{ttt}$ and replacing the linear term $\partial_{t}^{3}(-\partial_{x}^{3}u^{N})$ by $\partial_{t}^{4}u^{N}+P_{N}\partial_{t}^{3}(u^{N}u_{x}^{N})$ in view of (\ref{dd36}), we obtain, after some algebra, that
\begin{eqnarray}
A_{2}&=&\frac{b_{1}^{4}}{12}\partial_{t}^{4}u^{N}+\frac{b_{1}^{4}}{8}\left(-P_{N}\partial_{x}^{3}(u_{t}^{N}u_{xt}^{N})+2P_{N}(u_{t}^{N}u_{xtt}^{N}+u_{tt}^{N}u_{xt}^{N})\right.\nonumber\\
&&\left.-P_{N}\left(u^{N}P_{N}\left(u_{t}^{N}u_{xt}^{N}\right)\right)_{x}\right).\label{435}
\end{eqnarray}

\medskip
\noindent{(ii)}
\underline{Estimation of $e^{n,1}$}.

Having determined $A_{1}$ and $A_{2}$ we now equate the $O(k^{5})$ (and higher-order) terms in (\ref{428}) in order to find an equation for $e^{n,1}$. This gives
\begin{eqnarray*}
\rho_{1}+e^{n,1}&=&-kb_{1}\partial_{x}^{3}\rho_{2}-\frac{k^{5}b_{1}}{2}\partial_{x}^{3}A_{2}-\frac{kb_{1}}{2}\partial_{x}^{3}e^{n,1}\\
&&-kb_{1}P_{N}\left((\frac{u^{N}+u^{N}(\tau^{n,1})}{2})(\frac{u^{N}+u^{N}(\tau^{n,1})}{2})_{x}\right)\Big|_{O(k^{4})}\\
&&-\frac{k^{4}b_{1}}{4}P_{N}\left((\frac{u^{N}+u^{N}(\tau^{n,1})}{2})A_{1})\right)_{x}\Big|_{O(k)}\\
&&-\frac{k^{5}b_{1}}{4}P_{N}\left((\frac{u^{N}+u^{N}(\tau^{n,1})}{2})A_{2})\right)_{x}\\
&&-\frac{kb_{1}}{4}P_{N}\left(k^{6}A_{1}A_{1,x}+k^{7}(A_{1}A_{2})_{x}+k^{8}A_{2}A_{2,x}\right)-\frac{kb_{1}}{4}P_{N}\mathcal{A}(e^{n,1}),
\end{eqnarray*}
where $\rho_{1}, \rho_{2}$ satisfy (\ref{423}), (\ref{424}), respectively, and $\mathcal{A}(e^{n,1})$ is defined in (\ref{427}). The two terms denoted above as $\cdots \Big|_{O(k^{4})}, \cdots \Big|_{O(k)}$ will include Taylor remainders of the indicated order. We write the above equation as
\begin{eqnarray}
e^{n,1}+\frac{kb_{1}}{2}\partial_{x}^{3}e^{n,1}=\Gamma_{1}-\frac{kb_{1}}{4}P_{N}\mathcal{A}(e^{n,1}),\label{436}
\end{eqnarray}
where
\begin{eqnarray}
\Gamma_{1}&=&-\rho_{1}-kb_{1}\partial_{x}^{3}\rho_{2}-\frac{k^{5}b_{1}}{2}\partial_{x}^{3}A_{2}-kb_{1}P_{N}\left((\frac{u^{N}+u^{N}(\tau^{n,1})}{2})(\frac{u^{N}+u^{N}(\tau^{n,1})}{2})_{x}\right)\Big|_{O(k^{4})}\nonumber\\
&&-\frac{k^{4}b_{1}}{4}P_{N}\left((\frac{u^{N}+u^{N}(\tau^{n,1})}{2})A_{1})\right)_{x}\Big|_{O(k)}-\frac{k^{5}b_{1}}{4}P_{N}\left((\frac{u^{N}+u^{N}(\tau^{n,1})}{2})A_{2})\right)_{x}\nonumber\\
&&-\frac{kb_{1}}{4}P_{N}\left(k^{6}A_{1}A_{1,x}+k^{7}(A_{1}A_{2})_{x}+k^{8}A_{2}A_{2,x}\right).\label{437}
\end{eqnarray}
We shall prove below that for $\mu$ sufficiently large, there is a constant $C$, independent of $N$ and $k$, such that
\begin{eqnarray}
||\Gamma_{1}||\leq Ck^{5}.\label{438}
\end{eqnarray}
Assuming for the moment the validity of (\ref{438}), and taking inner products of both sides of (\ref{436}) with $e^{n,1}\in S_{N}$, we have, using integration by parts, periodicity, and (\ref{427}) that
\begin{eqnarray*}
||e^{n,1}||^{2}&=&(\Gamma_{1},e^{n,1})-\frac{kb_{1}}{8}((u^{N}+u^{N}(\tau^{n,1}))_{x}e^{n,1},e^{n,1})\\
&&-\frac{k^{4}b_{1}}{8}(A_{1,x}e^{n,1},e^{n,1})
-\frac{k^{5}b_{1}}{8}(A_{2,x}e^{n,1},e^{n,1}).
\end{eqnarray*}
Therefore
\begin{eqnarray}
||e^{n,1}||^{2}&\leq & ||\Gamma_{1}||||e^{n,1}||+Ck|(u^{N}+u^{N}(\tau^{n,1}))_{x}|_{\infty}||e^{n,1}||^{2}\nonumber\\
&&+Ck^{4}|A_{1,x}|_{\infty}||e^{n,1}||^{2}+Ck^{5}|A_{2,x}|_{\infty}||e^{n,1}||^{2}.\label{439}
\end{eqnarray}
Using Proposition \ref{propo32}, the fact that $|P_{N}v|_{\infty}\leq C||v||_{1}$, which follows from (\ref{dd22}) and Sobolev's theorem, and (\ref{431}), (\ref{435}), we see that for $0\leq t\leq T$,
\begin{eqnarray}
|(u^{N}+u^{N}(\tau^{n,1}))_{x}|_{\infty}&\leq & C\max_{t}||u^{N}||_{2}\leq C,\;{\rm for}\; \mu\geq 3,\label{440}\\
|A_{1,x}|_{\infty}&\leq & C, \;{\rm for}\; \mu\geq 12,\label{441}\\
|A_{2,x}|_{\infty}&\leq & C, \;{\rm for}\; \mu\geq 15.\label{442}
\end{eqnarray}
Therefore, using (\ref{430})-(\ref{442}) we see for $k$ sufficiently small that
\begin{eqnarray}
||e^{n,1}||\leq Ck^{5}, \;{\rm for}\; \mu\geq 15.\label{443}
\end{eqnarray}
We now prove (\ref{438}). We have
\begin{eqnarray*}
||\Gamma_{1}||&\leq & ||\rho_{1}||+Ck||\partial_{x}^{3}\rho_{2}||+Ck^{5}||\partial_{x}^{3}A_{2}||\\
&&+Ck^{5}|(u^{N}+u^{N}(\tau^{n,1}))_{x}|_{\infty}\max_{t}||\partial_{t}^{4}u^{N}||\\
&&+Ck^{5}||A_{1}||_{1,\infty}\max_{t}||\partial_{t}u^{N}||_{1}+Ck^{5}\max_{t}||u^{N}||_{1}||A_{2}||_{1,\infty}\\
&&+Ck^{5}\left(||A_{1}||_{1,\infty}^{2}+||A_{2}||_{1,\infty}^{2}\right).
\end{eqnarray*}
Therefore, by (\ref{423}), (\ref{424}), (\ref{431}), (\ref{435}), Proposition \ref{propo32} and estimates like (\ref{440})-(\ref{442}) we obtain
\begin{eqnarray}
||\Gamma_{1}||\leq Ck^{5}, \;{\rm for}\; \mu\geq 16,\label{444}
\end{eqnarray}
We note, for future use, that (\ref{436}) gives, in view of (\ref{427}), (\ref{440})-(\ref{444}) and (\ref{dd23}), for $\mu\geq 16$
\begin{eqnarray*}
k||\partial_{x}^{3}e^{n,1}||&\leq &C||\Gamma_{1}||+Ck||\mathcal{A}(e^{n,1})||+C||e^{n,1}||\\
&\leq & Ck^{5}+Ck||e^{n,1}||_{1}+Ck|e^{n,1}|_{\infty}||e_{x}^{n,1}||\\
&\leq &Ck^{5}+Ck^{5}(kN)+Ck^{5}(k^{6}N^{3/2}).
\end{eqnarray*}
Therefore, provided $k=O(N^{-1})$, we have
\begin{eqnarray}
k||\partial_{x}^{3}e^{n,1}||\leq Ck^{5}, \;{\rm if}\; \mu\geq 16, k=O(N^{-1}).\label{445}
\end{eqnarray}
For the estimates (\ref{443}) and (\ref{445}) we tracked, as an example, lower bounds of $\mu$ so that the constants involved are bounded. In the sequel we will just assume that $\mu$ is \lq sufficiently large\rq. Sufficient lower bounds of $\mu$ can always be retrieved if needed.

\medskip
\noindent{(iii)}
\underline{Asymptotic expansion of $V^{n,2}$. Determination of the coefficients $B_{1},B_{2}$}.

Using the Ansatz (\ref{419}) we now evaluate the $O(1)$ quantities $B_{1}$ and $B_{2}$. From (\ref{415}) for $i=2$ it follows that
\begin{eqnarray}
V^{n,2}&=&V^{n,1}-kb_{2}\left(\partial_{x}^{3}\left(\frac{V^{n,1}+V^{n,2}}{2}\right)\right)\nonumber\\
&&-\frac{kb_{2}}{4}P_{N}\left((V^{n,1}+V^{n,2})(V^{n,1}+V^{n,2})_{x}\right).\label{446}
\end{eqnarray}
We insert the expression (\ref{419}) in (\ref{446}); in the left-hand side, we obtain
\begin{eqnarray}
V^{n,2}
&=&u^{N}+k(b_{1}+b_{2})u_{t}^{N}+\frac{k^{2}(b_{1}+b_{2})^{2}}{2}u_{tt}^{N}+\frac{k^{3}(b_{1}+b_{2})^{3}}{6}u_{ttt}^{N}+\frac{k^{4}(b_{1}+b_{2})^{4}}{24}\partial_{t}^{4}u^{N}\nonumber\\
&&+\rho_{3}+B_{1}k^{3}+B_{2}k^{4}+e^{n,2},\label{447}
\end{eqnarray}
where the Taylor remainder $\rho_{3}$ satisfies
\begin{eqnarray}
||\rho_{3}||_{j}\leq Ck^{5}\max_{t}||\partial_{t}^{5}u^{N}||_{j}.\label{448}
\end{eqnarray}

Since in our case $b_{1}+b_{2}=1-b_{1}\cong -0.351$, in the first step of the fully discrete scheme $\tau^{0,2}=k(b_{1}+b_{2})$ will be negative. In addition, since $b_{1}>1$, for $n=M-1$ $\tau^{n,1}$ will exceed $T$.
Using the reversibility for $t<0$ of the KdV it is easy to see that $u^{N}(t)$ is defined for $t\in [-k,0]$ and satisfies the semidiscrete equations (\ref{dd24}) in $[-k,0]$. Obviously we may also extend the well-posedness of (\ref{dd12}) and the validity of (\ref{dd24}) up to $t=T+k$, as we have tacitly assumed in parts (i) and (ii) of the proof already. Hence the error estimate (\ref{dd31}) and the boundedness estimate (\ref{dd34}) are valid with the maximum taken over $[-k,T+k]$ now. In the sequel we will accordingly not specify the range of the subscripted $\max$ in formulas like (\ref{448}) as such estimates are obviously valid in the relevant intervals of $t$.

For the linear term in the right-hand side of (\ref{446}) we have, in view of (\ref{418}), (\ref{419}) and Taylor's theorem
\begin{eqnarray}
-kb_{2}\left(\partial_{x}^{3}\left(\frac{V^{n,1}+V^{n,2}}{2}\right)\right)
&=&-\frac{kb_{2}}{2}\partial_{x}^{3}\left(2u^{N}+k(2b_{1}+b_{2})u_{t}^{N}+\frac{k^{2}}{2}(b_{1}^{2}+(b_{1}+b_{2})^{2})u_{tt}^{N}\right.\nonumber\\
&&\left.+\frac{k^{3}}{6}(b_{1}^{3}+(b_{1}+b_{2})^{3})u_{ttt}^{N}+\rho_{4}+A_{1}k^{3}+A_{2}k^{4}\right.\nonumber\\
&&\left.+B_{1}k^{3}+B_{2}k^{4}+e^{n,1}+e^{n,2}\right),\label{449}
\end{eqnarray}
in which
\begin{eqnarray}
||\rho_{4}||_{j}\leq Ck^{4}\max_{t}||\partial_{t}^{4}u^{N}||_{j}.\label{450}
\end{eqnarray}
For the last term in the right-hand side of (\ref{446}) we have
\begin{eqnarray}
-\frac{kb_{2}}{4}P_{N}\left((V^{n,1}+V^{n,2})(V^{n,1}+V^{n,2})_{x}\right)
&=&-\frac{kb_{2}}{4}P_{N}\left((u^{N}(\tau^{n,1})+u^{N}(\tau^{n,2}))(u^{N}(\tau^{n,1})+u^{N}(\tau^{n,2}))_{x}\right.\nonumber\\
&&\left.+k^{3}\left((A_{1}+B_{1})(u^{N}(\tau^{n,1})+u^{N}(\tau^{n,2}))\right)_{x}\right.\nonumber\\
&&\left.+k^{4}\left((A_{2}+B_{2})(u^{N}(\tau^{n,1})+u^{N}(\tau^{n,2}))\right)_{x}\right.\nonumber\\
&&\left.+k^{6}(A_{1}+B_{1})(A_{1}+B_{1})_{x}+k^{7}\left((A_{1}+B_{1})(A_{2}+B_{2})\right)_{x}\right.\nonumber\\
&&\left.+k^{8}(A_{2}+B_{2})(A_{2}+B_{2})_{x}+\mathcal{B}(e^{n,1},e^{n,2})\right),\label{451}
\end{eqnarray}
where
\begin{eqnarray}
\mathcal{B}(e^{n,1},e^{n,2})&=&\left((u^{N}(\tau^{n,1})+u^{N}(\tau^{n,2}))(e^{n,1}+e^{n,2})\right)_{x}+k^{3}\left((A_{1}+B_{1})(e^{n,1}+e^{n,2})\right)_{x}\nonumber\\
&&+k^{4}\left((A_{2}+B_{2})(e^{n,1}+e^{n,2})\right)_{x}+
(e^{n,1}+e^{n,2})(e^{n,1}+e^{n,2})_{x}.\label{452}
\end{eqnarray}
From (\ref{446}), (\ref{447}), (\ref{422}), (\ref{449}), (\ref{451}) we have now
\begin{eqnarray}
u^{N}+k(b_{1}+b_{2})u_{t}^{N}+\frac{k^{2}(b_{1}+b_{2})^{2}}{2}u_{tt}^{N}&&
+\frac{k^{3}(b_{1}+b_{2})^{3}}{6}u_{ttt}^{N}+\frac{k^{4}(b_{1}+b_{2})^{4}}{24}\partial_{t}^{4}u^{N}\nonumber\\
+\rho_{3}+B_{1}k^{3}+B_{2}k^{4}+e^{n,2}&=&u^{N}+kb_{1}u_{t}^{N}+\frac{k^{2}b_{1}^{2}}{2}u_{tt}^{N}+\frac{k^{3}b_{1}^{3}}{6}u_{ttt}^{N}\nonumber\\
&&+\frac{k^{4}b_{1}^{4}}{24}\partial_{t}^{4}u^{N}+\rho_{1}+e^{n,1}+A_{1}k^{3}+A_{2}k^{4}\nonumber\\
&&-\frac{kb_{2}}{2}\partial_{x}^{3}\left(2u^{N}+k(2b_{1}+b_{2})u_{t}^{N}+\frac{k^{2}}{2}(b_{1}^{2}+(b_{1}+b_{2})^{2})u_{tt}^{N}\right.\nonumber\\
&&\left.+\frac{k^{3}}{6}(b_{1}^{3}+(b_{1}+b_{2})^{3})u_{ttt}^{N}+\rho_{4}+A_{1}k^{3}+A_{2}k^{4}\right.\nonumber\\
&&\left.+B_{1}k^{3}+B_{2}k^{4}+e^{n,1}+e^{n,2}\right)\nonumber\\
&&-\frac{kb_{2}}{4}P_{N}\left((u^{N}(\tau^{n,1})+u^{N}(\tau^{n,2}))(u^{N}(\tau^{n,1})+u^{N}(\tau^{n,2}))_{x}\right.\nonumber\\
&&\left.+k^{3}\left((A_{1}+B_{1})(u^{N}(\tau^{n,1})+u^{N}(\tau^{n,2}))\right)_{x}\right.\nonumber\\
&&\left.+k^{4}\left((A_{2}+B_{2})(u^{N}(\tau^{n,1})+u^{N}(\tau^{n,2}))\right)_{x}\right.\nonumber\\
&&\left.+k^{6}(A_{1}+B_{1})(A_{1}+B_{1})_{x}\right.\nonumber\\
&&\left.+k^{7}\left((A_{1}+B_{1})(A_{2}+B_{2})\right)_{x}\right.\nonumber\\
&&\left.+k^{8}(A_{2}+B_{2})(A_{2}+B_{2})_{x}\right.\nonumber\\
&&\left.+\mathcal{B}(e^{n,1},e^{n,2})\right),\label{453}
\end{eqnarray}
We now equate, as before, terms of the same power of $k$ in both sides of the above. (For this purpose we will need to expand some $u^{N}(\tau^{n,i})$ terms in the right hand-side  of (\ref{453}) in Taylor series about $t=t^{n}$.) It is straightforward to see that we get identities by equating the $O(1), O(k)$, and $O(k^{2})$ terms in both sides of (\ref{453}). For the identity of the $O(k)$ terms we have to use (\ref{dd36}) and for the one of the $O(k^{2})$ ones we need to differentiate both sides of (\ref{dd36}) with respect to $t$. These identities hold independently of the values of the $b_{i}$ as expected.

%
%
%
%

\medskip
{\it $O(k^{3})$ terms}:

From (\ref{453}) we get
\begin{eqnarray}
\frac{k^{3}}{6}(b_{1}+b_{2})^{3}u_{ttt}^{N}+k^{3}B_{1}&=&\frac{k^{3}}{6}b_{1}^{3}u_{ttt}^{N}+k^{3}A_{1}-\frac{k^{3}}{4}b_{2}(b_{1}^{2}+(b_{1}+b_{2})^{2})\partial_{x}^{3}u_{tt}^{N}\nonumber\\
&&-\frac{kb_{2}}{4}\Delta_{2},\label{454}
\end{eqnarray}
where
\begin{eqnarray*}
\Delta_{2}&=&P_{N}\left((u^{N}(\tau^{n,1})+u^{N}(\tau^{n,2}))(u^{N}(\tau^{n,1})+u^{N}(\tau^{n,2}))_{x}\Big|_{O(k^{2})}\right)\\
&=&k^{2}P_{N}\left((b_{1}^{2}+(b_{1}+b_{2})^{2})(u^{N}u_{ttx}^{N}+u_{x}^{N}u_{tt}^{N})+(2b_{1}+b_{2})^{2}u_{t}^{N}u_{tx}^{N}\right).
\end{eqnarray*}
Therefore, from (\ref{454})
\begin{eqnarray}
B_{1}&=&\frac{1}{6}(b_{1}^{3}-(b_{1}+b_{2})^{3})u_{ttt}^{N}+A_{1}-\frac{1}{4}b_{2}(b_{1}^{2}+(b_{1}+b_{2})^{2})\partial_{x}^{3}u_{tt}^{N}\nonumber\\
&&-\frac{1}{4}b_{2}P_{N}\left((b_{1}^{2}+(b_{1}+b_{2})^{2})(u^{N}u_{ttx}^{N}+u_{x}^{N}u_{tt}^{N}+2u_{t}^{N}u_{tx}^{N})\right.\nonumber\\
&&\left.-b_{2}^{2}u_{t}^{N}u_{tx}^{N}\right).\label{455}
\end{eqnarray}
Using Leibniz's rule for differentiation of products we have from (\ref{dd36})
\begin{eqnarray*}
u_{ttt}^{N}=-\partial_{x}^{3}u_{tt}^{N}-P_{N}(u_{tt}^{N}u_{x}^{N}+2u_{t}^{N}u_{xt}^{N}+u^{N}u_{xtt}^{N}).
\end{eqnarray*}
Therefore in (\ref{455}), using the facts that $b_{1}=b_{3}, b_{1}^{3}+b_{2}^{3}+b_{3}^{3}=0$, from which $b_{2}^{3}=-2b_{1}^{3}$, we see after some algebra and using (\ref{431}) that
\begin{eqnarray}
B_{1}=-A_{1}.\label{456}
\end{eqnarray}

{\it $O(k^{4})$ terms}:

From (\ref{453}), equating $O(k^{4})$ terms, and using appropriate Taylor expansions, and the fact that $A_{1}+B_{1}=0$, we obtain
\begin{eqnarray}
B_{2}=\frac{1}{24}(b_{1}^{4}-(b_{1}+b_{2})^{4})\partial_{t}^{4}u^{N}+A_{2}-\frac{b_{2}}{12}(b_{1}^{3}+(b_{1}+b_{2})^{3})\partial_{x}^{3}u_{ttt}^{N}-\frac{b_{2}}{4}\Delta_{3},\label{457}
\end{eqnarray}
where
\begin{eqnarray}
\Delta_{3}
&=&P_{N}\left(\frac{1}{3}(b_{1}^{3}+(b_{1}+b_{2})^{3})(u^{N}u_{tttx}^{N}+u_{ttt}^{N}u_{x}^{N})\right.\nonumber\\
&&\left.+\frac{1}{2}(b_{1}^{3}+b_{1}^{2}(b_{1}+b_{2})+b_{1}(b_{1}+b_{2})^{2}\right.\nonumber\\
&&\left.+(b_{1}+b_{2})^{3})(u_{}^{N}u_{ttx}^{N}+u_{tt}^{N}u_{tx}^{N})\right).\label{458}
\end{eqnarray}
Note that by differentiating (\ref{dd36}) three times with respect to $t$
and using Leibniz's rule for differentiation of the $u^{N}u_{x}^{N}$ term, we have
\begin{eqnarray}
-\partial_{x}^{3}\partial_{t}^{3}u^{N}=\partial_{t}^{4}u^{N}+P_{N}(u^{N}u_{xttt}^{N}+3u_{t}^{N}u_{xtt}^{N}+3u_{tt}^{N}u_{xt}^{N}+u_{ttt}^{N}u_{x}^{N}).\label{459}
\end{eqnarray}
To simplify somewhat (\ref{457}), using $b_{1}+b_{2}+b_{3}=1$ and $b_{1}=b_{3}$, i.~e. $2b_{1}+b_{2}=1$, we get
\begin{eqnarray}
\frac{b_{2}}{12}(b_{1}^{3}+(b_{1}+b_{2})^{3})=\frac{b_{2}}{12}(b_{1}^{2}+b_{1}b_{2}+b_{2}^{2}).\label{460}
\end{eqnarray}
From (\ref{460}), and (\ref{458}), and replacing in (\ref{457}) the term $-\partial_{x}^{3}u_{ttt}^{N}$ by the formula (\ref{459}) we see, 
after a number of algebraic computations using the facts that $2b_{1}+b_{2}=1, 2b_{1}^{3}+b_{2}^{3}=0$ that 
\begin{eqnarray}
B_{2}=A_{2}-\frac{b_{1}^{3}}{12}\partial_{t}^{4}u^{N}+\frac{b_{2}^{3}}{8}P_{N}\left(u_{t}^{N}u_{xtt}^{N}+u_{tt}^{N}u_{xt}^{N}\right).\label{461}
\end{eqnarray}

\medskip
\noindent{(iv)}
\underline{Estimation of $e^{n,2}$}.

Having determined the $O(1)$ quantities $B_{1}, B_{2}\in S_{N}$, we equate now the $O(k^{5})$ and higher-order terms in (\ref{453}) in order to find an equation for the residual $e^{n,2}$. This yields (if we use the fact that $A_{1}+B_{1}=0$)
\begin{eqnarray*}
\rho_{3}+e^{n,2}&=&\rho_{1}+e^{n,1}-\frac{kb_{2}}{2}\partial_{x}^{3}\left(\rho_{4}+(A_{2}+B_{2})k^{4}+e^{n,1}+e^{n,2}\right)\\
&&-\frac{kb_{2}}{4}P_{N}\left((u^{N}(\tau^{n,1})+u^{N}(\tau^{n,2}))(u^{N}(\tau^{n,1})+u^{N}(\tau^{n,2}))_{x}\Big|_{O(k^{4})}\right.\\
&&\left.+k^{4}\left((A_{2}+B_{2})(u^{N}(\tau^{n,1})+u^{N}(\tau^{n,2}))\Big|_{O(k)}\right)_{x}+k^{8}(A_{2}+B_{2})(A_{2}+B_{2})_{x}\right)\\
&&-\frac{kb_{2}}{4}P_{N}(\mathcal{B}(e^{n,1},e^{n,2})).
\end{eqnarray*}
We recall that $\rho_{1}, \rho_{3}, \rho_{4}$ satisfy (\ref{423}), (\ref{448}), (\ref{450}), respectively, and $\mathcal{B}(e^{n,1},e^{n,2})$ is given by (\ref{452}). The terms denoted as $\cdots\Big|_{O(k^{\alpha})}$ will include Taylor remainders of the indicated order. We simplify the above equation to
\begin{eqnarray}
e^{n,2}+\frac{kb_{2}}{2}\partial_{x}^{3}e^{n,2}=\Gamma_{2}+e^{n,1}-\frac{kb_{2}}{2}\partial_{x}^{3}e^{n,1}-\frac{kb_{2}}{4}P_{N}(\mathcal{B}(e^{n,1},e^{n,2})),\label{462}
\end{eqnarray}
where
\begin{eqnarray}
\Gamma_{2}&=&\rho_{1}-\rho_{3}-\frac{kb_{2}}{4}\partial_{x}^{3}\rho_{4}-\frac{k^{5}b_{2}}{2}\partial_{x}^{3}(A_{2}+B_{2})\nonumber\\
&&-\frac{kb_{2}}{4}P_{N}\left((u^{N}(\tau^{n,1})+u^{N}(\tau^{n,2}))(u^{N}(\tau^{n,1})+u^{N}(\tau^{n,2}))_{x}\Big|_{O(k^{4})}\right.\nonumber\\
&&\left.+k^{4}\left((A_{2}+B_{2})(u^{N}(\tau^{n,1})+u^{N}(\tau^{n,2}))\Big|_{O(k)}\right)_{x}\right.\nonumber\\
&&\left.+k^{8}(A_{2}+B_{2})(A_{2}+B_{2})_{x}\right).\label{463}
\end{eqnarray}
We shall prove below that for $\mu$ sufficiently large there is a constant $C$, independent of $N$ and $k$, such that
\begin{eqnarray}
||\Gamma_{2}||\leq Ck^{5}.\label{464}
\end{eqnarray}
Assuming for the time being the validity of (\ref{464}) and taking inner products in (\ref{462}) with $e^{n,2}$, we have, using (\ref{452}), periodicity, and the fact that $A_{1}+B_{1}=0$, that
\begin{eqnarray}
||e^{n,2}||^{2}&=&(\Gamma_{2},e^{n,2})+(e^{n,1},e^{n,2})-\frac{kb_{2}}{2}(\partial_{x}^{3}e^{n,1},e^{n,2})\nonumber\\
&&-\frac{kb_{2}}{4}\underbrace{\left(((u^{N}(\tau^{n,1})+u^{N}(\tau^{n,2}))(e^{n,1}+e^{n,2}))_{x},e^{n,2}\right)}_{I}\nonumber\\
&&-\frac{k^{5}b_{2}}{4}\underbrace{\left(((A_{2}+B_{2})(e^{n,1}+e^{n,2}))_{x},e^{n,2}\right)}_{II}\nonumber\\
&&-\frac{kb_{2}}{4}\underbrace{\left((e^{n,1}+e^{n,2})(e^{n,1}+e^{n,2})_{x},e^{n,2}\right)}_{III}.\label{465}
\end{eqnarray}
For the term I above, for $\mu$ sufficiently large, using integration by parts, and taking into account (\ref{dd34}), (\ref{dd23}), (\ref{443}), and the fact that $kN=O(1)$, we obtain

\begin{eqnarray}
|I|
&\leq &Ck||u^{N}(\tau^{n,1})+u^{N}(\tau^{n,2})||_{1,\infty}||e^{n,1}||||e^{n,2}||+Ck||u^{N}(\tau^{n,1})+u^{N}(\tau^{n,2})||_{1,\infty}||e^{n,2}||^{2}\nonumber\\
&&+Ck|u^{N}(\tau^{n,1})+u^{N}(\tau^{n,2})|_{\infty}||e_{x}^{n,1}||||e^{n,2}||\nonumber
\\
&\leq & Ck||e^{n,1}||||e^{n,2}||+Ck||e^{n,2}||^{2}+Ck||e_{x}^{n,1}||||e^{n,2}||\nonumber\\
&\leq & Ck||e^{n,1}||||e^{n,2}||+CkN||e^{n,1}||||e^{n,2}||+Ck||e^{n,2}||^{2}\nonumber
\\
&\leq & Ck^{5}||e^{n,2}||+Ck||e^{n,2}||^{2}.\label{466}
\end{eqnarray}
To estimate II, note that by (\ref{435}), (\ref{442}), (\ref{461}), and (\ref{dd34}), for $\mu$ sufficiently large, we obtain $||A_{2}+B_{2}||_{1,\infty}\leq C$. Hence, using integration by parts, (\ref{443}), (\ref{dd23}), and taking $\mu$ sufficiently large, and using the fact that $kN=O(1)$, we have
\begin{eqnarray}
|II|
&\leq & Ck^{5}|(A_{2}+B_{2})_{x}|_{\infty}\left(||e^{n,1}||||e^{n,2}||+||e^{n,2}||^{2}\right)\nonumber\\
&&+Ck^{5}|(A_{2}+B_{2})|_{\infty}||e_{x}^{n,1}||||e^{n,2}||
\nonumber\\
&\leq &
Ck^{9}||e^{n,2}||+Ck^{5}||e^{n,2}||^{2}.\label{467}
\end{eqnarray}
For the term III of (\ref{465}), using integration by parts, (\ref{443}), (\ref{dd23}), and $kN=O(1)$, we see that
\begin{eqnarray}
|III|&\leq &  Ck|e^{n,1}|_{\infty}||e_{x}^{n,1}||||e^{n,2}||+Ck|e^{n,1}|_{\infty}||e^{n,2}||^{2}\nonumber
\\
&\leq & Ck^{11}N^{3/2}||e^{n,2}||+Ck^{6}N^{3/2}||e^{n,2}||^{2}\nonumber
\\
&\leq &
Ck^{9.5}||e^{n,2}||+Ck^{4.5}||e^{n,2}||^{2}.\label{468}
\end{eqnarray}
From (\ref{466})-(\ref{468}) we conclude, for $kN=O(1)$, $\mu$ sufficiently large, that
\begin{eqnarray}
|I+II+III|\leq Ck^{5}||e^{n,2}||+Ck||e^{n,2}||^{2}.\label{469}
\end{eqnarray}
Hence, by (\ref{465}) and (\ref{469}) we have,  for $kN=O(1)$, $\mu$ sufficiently large
\begin{eqnarray*}
||e^{n,2}||^{2}&\leq & ||\Gamma_{2}||||e^{n,2}||+||e^{n,1}||||e^{n,2}||+Ck||\partial_{x}^{3}e^{n,1}||||e^{n,2}||\\
&&+ Ck^{5}||e^{n,2}||+Ck||e^{n,2}||^{2}.
\end{eqnarray*}
From (\ref{443}), (\ref{445}), (\ref{464}), and the above, we conclude therefore, for $kN=O(1)$, $\mu$ sufficiently large and $k$ sufficiently small, that
\begin{eqnarray}
||e^{n,2}||\leq Ck^{5}.\label{470}
\end{eqnarray}
As done for $e^{n,1}$, it turns out that we will need in the sequel an optimal-order estimate for $k||\partial_{x}^{3}e^{n,2}||$ under no prohibitive stability assumptions. For this purpose, note that (\ref{462}) yields
\begin{eqnarray}
k||\partial_{x}^{3}e^{n,2}||\leq  C\left(||\Gamma_{2}||+||e^{n,1}||+||e^{n,2}||+k||\partial_{x}^{3}e^{n,1}||+k||\mathcal{B}(e^{n,1},e^{n,2})||\right).\label{471}
\end{eqnarray}
Now, for $\mu$ sufficiently large, we have from (\ref{452}), using similar estimates as before,
\begin{eqnarray*}
k||\mathcal{B}(e^{n,1},e^{n,2})||&\leq  &Ck||e^{n,1}+e^{n,2}||_{1}+Ck^{4}||e^{n,1}+e^{n,2}||_{1}\\
&&+Ck|e^{n,1}+e^{n,2}|_{\infty}||e^{n,1}+e^{n,2}||_{1}.
\end{eqnarray*}
Hence using (\ref{443}), (\ref{470}), (\ref{dd23}), and $kN=O(1)$ we see that
\begin{eqnarray}
k||\mathcal{B}(e^{n,1},e^{n,2})||\leq Ck^{5}.\label{472} 
\end{eqnarray}
It follows from (\ref{471}), (\ref{464}), (\ref{443}), (\ref{470}), (\ref{445}), (\ref{472}), that
\begin{eqnarray}
k||\partial_{x}^{3}e^{n,2}||\leq Ck^{5}.\label{473} 
\end{eqnarray}
We finally establish (\ref{464}). Using Taylor expansions to the required order, we have for $\mu$ sufficiently large
\begin{eqnarray*}
||\Gamma_{2}||&\leq & ||\rho_{1}||+||\rho_{3}||+Ck||\partial_{x}^{3}\rho_{4}||+Ck^{5}||\partial_{x}^{3}(A_{2}+B_{2})||\\
&&+Ck(Ck^{4}\max_{t}||\partial_{t}^{4}u^{N}||)+Ck\left(||A_{2}+B_{2}||_{1,\infty}\cdot Ck\max_{t}||u_{t}^{N}||_{1}\right)\\
&&+Ck^{9}|A_{2}+B_{2}|_{\infty}||(A_{2}+B_{2})_{x}||\leq Ck^{5},
\end{eqnarray*}
where we used (\ref{423}), (\ref{448}), (\ref{450}), (\ref{435}), (\ref{461}) in conjunction with (\ref{dd34}). Therefore (\ref{464}) holds.

\medskip
\noindent{(v)}
\underline{Final consistency step: Verify that (\ref{420}) holds with $e^{n,3}$ satisfying $||e^{n,3}||\leq Ck^{5}$}.

In this final step we let $e^{n,3}\in S_{N}$ be defined by (\ref{420}), find a suitable equation for $e^{n,3}$ (as we did for $e^{n,1}$ and $e^{n,2}$) and prove that $||e^{n,3}||\leq Ck^{5}$. For this purpose
we substitute (\ref{420}) in the equation for $V^{n,3}$ in (\ref{415}) and prove that $||e^{n,3}||\leq Ck^{5}$, using the expansion (\ref{419}) for $V^{n,2}$ and the estimates that we have for $B_{i}, e^{n,2}$.

Substituting $V^{n,3}$ from (\ref{420})  in (\ref{415}) and using (\ref{419}) we get
\begin{eqnarray}
u^{N}(t^{n+1})+e^{n,3}&=&u^{N}(\tau^{n,2})+B_{1}k^{3}+B_{2}k^{4}+e^{n,2}-\frac{kb_{3}}{2}\partial_{x}^{3}(u^{N}(\tau^{n,2})+u^{N}(t^{n+1}))\nonumber\\
&&-\frac{kb_{3}}{2}\partial_{x}^{3}e^{n,3}-\frac{kb_{3}}{2}\partial_{x}^{3}(B_{1}k^{3}+B_{2}k^{4}+e^{n,2})\nonumber\\
&&-\frac{kb_{3}}{4}P_{N}\left((u^{N}(\tau^{n,2})+u^{N}(t^{n+1}))(u^{N}(\tau^{n,2})+u^{N}(t^{n+1}))_{x}\right.\nonumber\\
&&\left. +k^{3}\left((u^{N}(\tau^{n,2})+u^{N}(t^{n+1}))B_{1}\right)_{x}+k^{4}\left((u^{N}(\tau^{n,2})+u^{N}(t^{n+1}))B_{2}\right)_{x}\right.\nonumber\\
&&\left.+k^{6}B_{1}B_{1x}+k^{7}(B_{1}B_{2})_{x}+k^{8}(B_{2}B_{2x})+\mathcal{E}(e^{n,2},e^{n,3})\right),\label{474}
\end{eqnarray}
where
\begin{eqnarray}
\mathcal{E}(e^{n,2},e^{n,3})&=&\left((u^{N}(\tau^{n,2})+u^{N}(t^{n+1}))(e^{n,2}+e^{n,3})\right)_{x}+k^{3}\left(B_{1}(e^{n,2}+e^{n,3})\right)_{x}\nonumber\\
&&+k^{4}\left(B_{2}(e^{n,2}+e^{n,3})\right)_{x}+(e^{n,2}+e^{n,3})(e^{n,2}+e^{n,3})_{x}.\label{475}
\end{eqnarray}
Using now Taylor expansions in the linear terms of (\ref{474}) we have
\begin{eqnarray}
&&u^{N}+ku_{t}^{N}+\frac{k^{2}}{2}u_{tt}^{N}+\frac{k^{3}}{6}u_{ttt}^{N}+\frac{k^{4}}{24}\partial_{t}^{4}u^{N}+\rho_{5}+e^{n,3}\nonumber\\
&=&u^{N}+k(b_{1}+b_{2})u_{t}^{N}+\frac{k^{2}}{2}(b_{1}+b_{2})^{2}u_{tt}^{N}+\frac{k^{3}}{6}(b_{1}+b_{2})^{3}u_{ttt}^{N}+\frac{k^{4}}{24}(b_{1}+b_{2})^{4}\partial_{t}^{4}u^{N}+\rho_{6}\nonumber\\
&&+B_{1}k^{3}+B_{2}k^{4}+e^{n,2}-\frac{kb_{3}}{2}\partial_{x}^{3}\left(2u^{N}+k(b_{1}+b_{2}+1)u_{t}^{N}+\frac{k^{2}}{2}((b_{1}+b_{2})^{2}+1)u_{tt}^{N}\right.\nonumber\\
&&\left. +\frac{k^{3}}{6}((b_{1}+b_{2})^{3}+1)u_{ttt}^{N}+\rho_{7}\right)-\frac{kb_{3}}{2}\partial_{x}^{3}e^{n,3}-\frac{kb_{3}}{2}\partial_{x}^{3}(B_{1}k^{3}+B_{2}k^{4}+e^{n,2})\nonumber\\
&&-\frac{kb_{3}}{4}P_{N}\left((u^{N}(\tau^{n,2})+u^{N}(t^{n+1}))(u^{N}(\tau^{n,2})+u^{N}(t^{n+1}))_{x}\right.\nonumber\\
&&\left. +k^{3}\left((u^{N}(\tau^{n,2})+u^{N}(t^{n+1}))B_{1}\right)_{x}+k^{4}\left((u^{N}(\tau^{n,2})+u^{N}(t^{n+1}))B_{2}\right)_{x}\right.\nonumber\\
&&\left.+k^{6}B_{1}B_{1x}+k^{7}(B_{1}B_{2})_{x}+k^{8}(B_{2}B_{2x})+\mathcal{E}(e^{n,2},e^{n,3})\right).\label{476}
\end{eqnarray}
Here the residuals $\rho_{5}, \rho_{6}, \rho_{7}\in S_{N}$ satisfy
\begin{eqnarray}
||\rho_{5}||+||\rho_{6}|| \leq Ck^{5}\max_{t}||\partial_{t}^{5}u^{N}||,\;\;
||\rho_{7}||_{j}\leq Ck^{4}\max_{t}||\partial_{t}^{4}u^{N}||_{j}.\label{477}
\end{eqnarray}
We equate now equal-power terms in (\ref{476}). It is evident that the $O(1)$ terms give an identity. It is also straightforward to see that we get identities for the $O(k)$ and $O(k^{2})$ terms, using the facts that $b_{1}+b_{2}+b_{3}=1, b_{1}=b_{3}$, and (\ref{dd36}) and its temporal derivative, respectively.

{\it $O(k^{3})$ terms}:

From (\ref{476}), using (\ref{456}), (\ref{431}), and a Taylor expansion up to $O(k^{2})$ terms in the first nonlinear term in the right-hand side of (\ref{476}), we see that we have to check whether
\begin{eqnarray}
\frac{1}{6}u_{ttt}^{N}&=&\frac{1}{6}(b_{1}+b_{2})^{3}u_{ttt}^{N}-\frac{b_{1}^{3}}{12}u_{ttt}^{N}-\frac{b_{1}^{3}}{4}P_{N}(u_{t}^{N}u_{xt}^{N})-\frac{b_{3}}{4}((b_{1}+b_{2})^{2}+1)\partial_{x}^{3}u_{tt}^{N}\nonumber\\
&&-\frac{b_{3}}{4}P_{N}\left((b_{1}+b_{2})^{2}+1)(u^{N}u_{xtt}^{N}+u_{tt}^{N}u_{x}^{N})+(b_{1}+b_{2}+1)^{2}u_{t}^{N}u_{tx}^{N}\right).\label{478}
\end{eqnarray}
Differentiating (\ref{dd36}) twice with respect to $t$ and using Leibniz's rule we get

\begin{eqnarray*}
-\partial_{x}^{3}u_{tt}^{N}&=&\partial_{t}^{3}u^{N}+P_{N}\left(u^{N}u^{N}_{xtt}+2u_{t}^{N}u_{xt}^{N}+u_{tt}^{N}u_{x}^{N}\right).
\end{eqnarray*}
If we insert this expression for $-\partial_{x}^{3}u_{tt}^{N}$ in the fourth term in the right-hand side of (\ref{478}) and use the facts that $b_{1}=b_{3}, b_{1}+b_{2}+b_{3}=1$, we may see after some algebra that (\ref{478}) holds.

{\it $O(k^{4})$ terms}:

From (\ref{461}), (\ref{435}), (\ref{456}), and Taylor expansions in the first and second nonlinear terms in the right-hand side of (\ref{476}) we see that we must verify whether
\begin{eqnarray}
\frac{1}{24}\partial_{t}^{4}u^{N}&=&\frac{1}{24}(b_{1}+b_{2})^{4}\partial_{t}^{4}u^{N}-\frac{1}{12}b_{1}^{3}\partial_{t}^{4}u^{N}+\frac{b_{1}^{4}}{12}\partial_{t}^{4}u^{N}\nonumber\\
&&+\frac{b_{1}^{4}}{8}\left(-P_{N}\partial_{x}^{3}(u_{t}^{N}u_{xt}^{N})+2P_{N}(u_{t}^{N}u_{xtt}^{N}+u_{tt}^{N}u_{xt}^{N})-P_{N}\left(u^{N}P_{N}(u_{t}^{N}u_{xt}^{N}))_{x}\right)\right)\nonumber\\
&&+\frac{b_{2}^{3}}{8}P_{N}(u_{t}^{N}u_{ttx}^{N}+u_{tt}^{N}u_{tx}^{N})-\frac{b_{3}}{12}((b_{1}+b_{2})^{3}+1)\partial_{x}^{3}u_{ttt}^{N}\nonumber\\
&&+\frac{b_{3}b_{1}^{3}}{2}\left(\frac{1}{12}\partial_{x}^{3}\partial_{t}^{3}u^{N}+\frac{1}{4}\partial_{x}^{3}P_{N}(u_{t}^{N}u_{xt}^{N})\right)\nonumber\\
&&-\frac{b_{3}}{4}P_{N}\left(\frac{1}{3}((b_{1}+b_{2})^{3}+1)(u^{N}u_{xttt}^{N}+u_{ttt}^{N}u_{x}^{N})\right.\nonumber\\
&&\left. +\frac{1}{2}\left(1+(b_{1}+b_{2})+(b_{1}+b_{2})^{2}+(b_{1}+b_{2})^{3}\right)(u_{t}^{N}u_{xtt}^{N}+u_{tt}^{N}u_{xt}^{N})\right)\nonumber\\
&&+\frac{b_{3}}{2}P_{N}\left(\frac{b_{1}^{3}}{12}(u^{N}u_{ttt}^{N})_{x}+\frac{b_{1}^{3}}{4}\left(u^{N}P_{N}(u_{t}^{N}u_{xt}^{N})\right)_{x}\right).\label{480}
\end{eqnarray}
Let $L$ be the sum of the linear terms and $N_{1}$ the sum of the nonlinear terms in the right-hand side of (\ref{480}). Then
\begin{eqnarray}
L=\gamma_{1}\partial_{t}^{4}u^{N}+\gamma_{2}\partial_{x}^{3}u_{ttt}^{N},\label{481}
\end{eqnarray}
where
\begin{eqnarray}
\gamma_{1}=\frac{1}{24}(b_{1}+b_{2})^{4}-\frac{b_{1}^{3}}{12}+\frac{b_{1}^{4}}{12},\; \gamma_{2}=-\frac{b_{3}}{12}((b_{1}+b_{2})^{3}+1)+\frac{b_{1}^{3}b_{3}}{24}.\label{482}
\end{eqnarray}
Differentiating now (\ref{dd36}) three times with respect to $t$ we obtain
$$-\partial_{x}^{3}\partial_{t}^{3}u^{N}=\partial_{t}^{4}u^{N}+P_{N}(u^{N}u_{x}^{N})_{ttt}.$$ Therefore (\ref{481}) becomes
\begin{eqnarray}
L=(\gamma_{1}-\gamma_{2})\partial_{t}^{4}u^{N}-\gamma_{2}P_{N}(u^{N}u_{x}^{N})_{ttt},\label{483}
\end{eqnarray}
and we see that $L$ has now acquired a nonlinear term as a result of eliminating $\partial_{x}^{3}u_{ttt}^{N}$. Taking into account that $b_{1}=b_{3}, b_{2}=1-2b_{1}$ we see after some algebra that
\begin{eqnarray}
\gamma_{1}-\gamma_{2}=\frac{1}{24}.\label{484}
\end{eqnarray}
Using (\ref{484}) and (\ref{483}) we may check that the $\partial_{t}^{4}u^{N}$ terms in the two sides of (\ref{480}) match. Hence, in order to show that (\ref{480}) holds, in view of (\ref{483}) we have to check that
\begin{eqnarray}
N_{1}-\gamma_{2}P_{N}(u^{N}u_{x}^{N})_{ttt}=0,\label{485}
\end{eqnarray}
where we recall that $N_{1}$ is the sum of the original nonlinear terms in the right-hand side of (\ref{480}).
It holds that
\begin{eqnarray}
N_{1}-\gamma_{2}P_{N}(u^{N}u_{x}^{N})_{ttt}=P_{N}\mathcal{G},\label{486}
\end{eqnarray}
where 
\begin{eqnarray*}
\mathcal{G}&=&-\frac{b_{1}^{4}}{8}\partial_{x}^{3}(u_{t}^{N}u_{xt}^{N})+\frac{b_{1}^{4}}{4}(u_{t}^{N}u_{xtt}^{N}+u_{tt}^{N}u_{xt}^{N})-\frac{b_{1}^{4}}{8}\left(u^{N}P_{N}(u_{t}^{N}u_{xt}^{N})\right)_{x}\\
&&+\frac{b_{2}^{3}}{8}(u_{t}^{N}u_{xtt}^{N}+u_{tt}^{N}u_{xt}^{N})+\frac{b_{1}^{3}b_{3}}{8}\partial_{x}^{3}(u_{t}^{N}u_{xt}^{N})\\
&&-\frac{b_{3}}{12}((b_{1}+b_{2})^{3}+1)(u^{N}u_{xttt}^{N}+u_{ttt}^{N}u_{x}^{N})\\
&&-\frac{b_{3}}{8}\left(1+(b_{1}+b_{2})+(b_{1}+b_{2})^{2}+(b_{1}+b_{2})^{3}\right)(u_{t}^{N}u_{xtt}^{N}+u_{tt}^{N}u_{xt}^{N})\\
&&+\frac{b_{1}^{3}b_{3}}{24}(u^{N}u_{ttt}^{N})_{x}+\frac{b_{1}^{3}b_{3}}{8}\left(u^{N}P_{N}(u_{t}^{N}u_{xt}^{N})\right)_{x}-\gamma_{2}(u^{N}u_{x}^{N})_{ttt}.
\end{eqnarray*}
Since $b_{1}=b_{3}$ we see that the strange terms $\partial_{x}^{3}(u_{t}^{N}u_{xt}^{N})$ and $\left(u^{N}P_{N}(u_{t}^{N}u_{xt}^{N})\right)_{x}$ cancel, and we are left with (after some algebra and the application of Leibniz's rule)
\begin{eqnarray}
\mathcal{G}=\gamma_{3}(u^{N}u_{xttt}^{N}+u_{ttt}^{N}u_{x}^{N})+\gamma_{4}(u_{t}^{N}u_{xtt}^{N}+u_{tt}^{N}u_{xt}^{N}),\label{487}
\end{eqnarray}
where
\begin{eqnarray*}
\gamma_{3}&=&-\frac{b_{3}}{12}((b_{1}+b_{2})^{3}+1)+\frac{b_{1}^{3}b_{3}}{24}-\gamma_{2},\\
\gamma_{4}&=&\frac{b_{1}^{4}}{4}+\frac{b_{2}^{3}}{8}-\frac{b_{3}}{8}\left(1+(b_{1}+b_{2})+(b_{1}+b_{2})^{2}+(b_{1}+b_{2})^{3}\right)-3\gamma_{2}.
\end{eqnarray*}
From (\ref{482}) we see that
\begin{eqnarray}
\gamma_{3}=0.\label{488}
\end{eqnarray}
Finally, after some algebra and using the facts $b_{1}=b_{3}, b_{1}+b_{2}+b_{3}=1$ and that $x=b_{1}$ satisfies the cubic equation $x^{3}-2x^{2}+x-1/6=0$, we see that
\begin{eqnarray}
\gamma_{4}=0.\label{489}
\end{eqnarray}
Hence, (\ref{488}) and (\ref{489}) yield that $\mathcal{G}=0$ in (\ref{487}) and therefore, in view of (\ref{486}), (\ref{485}) holds. We conclude that (\ref{480}) is valid.

We now embark upon finding an equation for $e^{n,3}$ from the remaining terms in (\ref{476}). We recall that we have used Taylor expansions in the first two terms of the nonlinear $-\frac{kb_{3}}{4}P_{N}(\cdots)$ term in the right-hand side of (\ref{476}) and we have now to put in the residuals. In this way (\ref{476}) yields
\begin{eqnarray}
e^{n,3}&=&-\rho_{5}+\rho_{6}+e^{n,2}-\frac{kb_{3}}{2}\partial_{x}^{3}\rho_{7}-\frac{kb_{3}}{2}\partial_{x}^{3}e^{n,3}-\frac{k^{5}b_{3}}{2}\partial_{x}^{3}B_{2}\nonumber\\
&&-\frac{kb_{3}}{2}\partial_{x}^{3}e^{n,2}-\frac{kb_{3}}{4}P_{N}\left(\rho_{8}+k^{3}(\rho_{9}B_{1})_{x}+k^{4}\left((u^{N}(\tau^{n,2})+u^{N}(t^{n+1}))B_{2}\right)_{x}\right.\nonumber\\
&&\left.+k^{6}B_{1}B_{1x}+k^{7}(B_{1}B_{2})_{x}+k^{8}(B_{2}B_{2x})+\mathcal{E}(e^{n,2},e^{n,3})\right),\label{490}
\end{eqnarray}
where, using (\ref{dd34}), we see that
\begin{eqnarray}
||\rho_{8}||\leq Ck^{4}|u^{N}(\tau^{n,2})+u^{N}(t^{n+1})|_{\infty}\max_{t}||\partial_{t}^{4}u_{x}^{N}||\leq Ck^{4},\label{491}
\end{eqnarray}
for $\mu$ sufficiently large. Similarly
\begin{eqnarray}
||\rho_{9}||\leq Ck\max_{t}||\partial_{t}u^{N}||_{1}\leq Ck.\label{492}
\end{eqnarray}
Therefore, (\ref{490}) gives
\begin{eqnarray}
e^{n,3}+\frac{kb_{3}}{2}\partial_{x}^{3}e^{n,3}&=&\Gamma_{3}+e^{n,2}-\frac{kb_{3}}{2}\partial_{x}^{3}e^{n,2}-\frac{kb_{3}}{4}P_{N}\mathcal{E}(e^{n,2},e^{n,3}),\label{493}
\end{eqnarray}
where
\begin{eqnarray}
\Gamma_{3}&=&-\rho_{5}+\rho_{6}-\frac{kb_{3}}{2}\partial_{x}^{3}\rho_{7}-\frac{k^{5}b_{3}}{2}\partial_{x}^{3}B_{2}\nonumber\\
&&-\frac{kb_{3}}{4}P_{N}\left(\rho_{8}+k^{3}(\rho_{9}B_{1})_{x}+k^{4}\left((u^{N}(\tau^{n,2})+u^{N}(t^{n+1}))B_{2}\right)_{x}\right.\nonumber\\
&&\left.+k^{6}B_{1}B_{1x}+k^{7}(B_{1}B_{2})_{x}+k^{8}(B_{2}B_{2x})\right).\label{494}
\end{eqnarray}
We will prove below that for $\mu$ sufficiently large there is a constant $C$, independent of $k$ and $N$, such that
\begin{eqnarray}
||\Gamma_{3}||\leq C k^{5}.\label{495}
\end{eqnarray}
Assuming (\ref{495}) for the time being and taking $L^{2}$ inner products with $e^{n,3}\in S_{N}$ in (\ref{493}) we see, using integration by parts, that, in view of (\ref{475}),
\begin{eqnarray}
||e^{n,3}||^{2}
&=&(\Gamma_{3},e^{n,3})+(e^{n,2},e^{n,3})-\frac{kb_{3}}{2}(\partial_{x}^{3}e^{n,2},e^{n,3})\nonumber\\
&&-\underbrace{\frac{kb_{3}}{4}\left(\left((u^{N}(\tau^{n,2})+u^{N}(t^{n+1}))(e^{n,2}+e^{n,3})\right)_{x},e^{n,3}\right)}_{I}\nonumber\\
&&-\underbrace{\frac{k^{4}b_{3}}{4}\left(\left(B_{1}(e^{n,2}+e^{n,3})\right)_{x},e^{n,3}\right)}_{II}
-\underbrace{\frac{k^{5}b_{3}}{4}\left(\left(B_{2}(e^{n,2}+e^{n,3})\right)_{x},e^{n,3}\right)}_{III}\nonumber\\
&&-\underbrace{\frac{kb_{3}}{4}\left((e^{n,2}+e^{n,3})(e^{n,2}+e^{n,3})_{x},e^{n,3}\right)}_{IV}.\label{496}
\end{eqnarray}
We estimate first the terms I-IV in (\ref{496}).
By integration by parts, using (\ref{470}) and (\ref{dd23}), taking $\mu$ sufficiently large, and using our hypothesis that $kN=O(1)$, we obtain
\begin{eqnarray}
|I|
\leq   Ck^{5}||e^{n,3}||+Ck||e^{n,3}||^{2}.\label{497}
\end{eqnarray}
Now, by integrating by parts, using (\ref{456}), (\ref{441}), (\ref{470}), (\ref{dd23}), and taking $\mu$ sufficiently large, we see, in view of our hypothesis $kN=O(1)$, that
\begin{eqnarray}
|II|\leq  Ck^{8}||e^{n,3}||+Ck^{4}||e^{n,3}||^{2}.\label{498}
\end{eqnarray}
By integration by parts, (\ref{461}), (\ref{462}), (\ref{470}), (\ref{dd23}), and the fact that $kN=O(1)$, we have for $\mu$ sufficiently large,
\begin{eqnarray}
|III|\leq  Ck^{9}||e^{n,3}||+Ck^{5}||e^{n,3}||^{2}.\label{499}
\end{eqnarray}
Finally, using integration by parts, (\ref{470}), (\ref{dd23}), and the fact that $kN=O(1)$, we get 
\begin{eqnarray}
|IV|\leq Ck^{9}||e^{n,3}||+Ck^{4.5}||e^{n,3}||^{2}.\label{4100}
\end{eqnarray}
From (\ref{495})-(\ref{4100}) we conclude for $\mu$ sufficiently large and $k=O(N^{-1})$ that
\begin{eqnarray*}
||e^{n,3}||^{2}&\leq & Ck^{5}||e^{n,3}||+||e^{n,2}||||e^{n,3}||+Ck||\partial_{x}^{3}e^{n,2}||||e^{n,3}||+Ck||e^{n,3}||^{2}.
\end{eqnarray*}
Therefore, using (\ref{470}) and (\ref{473}), for $\mu$ sufficiently large, and the fact that $k=O(N^{-1})$, and taking $k$ sufficiently small, we finally obtain
\begin{eqnarray}
||e^{n,3}||\leq  Ck^{5},\label{4101}
\end{eqnarray}
To conclude the proof we have to check (\ref{495}). This is not hard to verify, in view of (\ref{494}), (\ref{477}), (\ref{435}), (\ref{461}), (\ref{491}), (\ref{492}), (\ref{456}), (\ref{dd34}), assuming as usual that $\mu$ is sufficiently large. Therefore, by (\ref{4101}) and (\ref{420}) the proof of Proposition \ref{propo43} is now complete.
\end{proof}

\section{Error estimate for the fully discrete scheme}
\label{sec5}
In this section we will consider again the fully discrete scheme given by (\ref{411}) or (\ref{413}) and corresponding to the temporal discretization of (\ref{dd24}) by the general IMR-based, $s-$stage RK-composition method given by (\ref{44}) or (\ref{45}), and prove, under certain conditions on the discretization parameters and provided the solution of (\ref{dd12}) belongs to $H^{\mu}$ for $\mu$ sufficiently large, that it has a unique solution $U^{n}$ satisfying the $L^{2}$ error estimate
\begin{eqnarray*}
\max_{0\leq n\leq M}||U^{n}-u(t^{n})||\leq C\left(k^{\alpha}+N^{1-\mu}\right),
\end{eqnarray*}
where $\alpha$ is a positive integer, and the local temporal error, defined by an analogous formula to (\ref{416}), will be {\it assumed} to be of $O(k^{\alpha+1})$ in $L^{2}$. (In section \ref{sec44} we considered the special case corresponding to $s=3$ and constants $b_{i}$ given by (\ref{46}) and proved that for that scheme $\alpha$ was equal to $4$ provided $\mu$ was sufficiently large and $kN=O(1)$.)

We first establish notation and present a summary of the main steps of the proof. For convenience in referencing we rewrite here the scheme (\ref{413}). We seek $Y^{n,i}, 0\leq i\leq s$, and $U^{n}$, $0\leq n\leq M$, in $S_{N}$, such that for $0\leq n\leq M-1$
\begin{eqnarray}
Y^{n,0}&=&U^{n},\nonumber\\
Y^{n,i}&=&Y^{n,i-1} +{kb_{i}}{F}\left(\frac{Y^{n,i}+Y^{n,i-1}}{2}\right),\quad 1\leq i\leq s,\nonumber\\
U^{n+1}&=&Y^{n,s},\label{51}
\end{eqnarray} 
and
\begin{eqnarray*}
U^{0}=P_{N}u_{0}.
\end{eqnarray*}
(Recall that for $v\in S_{N}, F(v)=-v_{xxx}-P_{N}f(v)_{x}$, where $f(v)=v^{2}/2$.)
As in section \ref{sec44} we define the local temporal error of the scheme (\ref{51}) in terms of the semidiscrete approximation $u^{N}$. For this purpose we write for $0\leq n\leq M$ $V^{n}=u^{N}(t^{n})$ and define $V^{n,i}\in S_{N}$ for $0\leq i\leq s, 0\leq n\leq M-1$, by
\begin{eqnarray}
V^{n,0}&=&V^{n},\label{52}\\
V^{n,i}&=&V^{n,i-1} +{kb_{i}}{F}\left(\frac{V^{n,i}+V^{n,i-1}}{2}\right),\quad 1\leq i\leq s,\nonumber
\end{eqnarray} 
and the local temporal error $\theta^{n}\in S_{N}, 0\leq n\leq M-1$, as
\begin{eqnarray}
\theta^{n}=V^{n+1}-V^{n,s}\equiv u^{N}(t^{n+1})-V^{n,s}.\label{53}
\end{eqnarray}
(We use the same notation for $V^{n}, V^{n,i}, \theta^{n}$ as in Section \ref{sec44} as no confusion will arise.)
For the local error we will assume that
\begin{eqnarray}
\max_{0\leq n\leq M-1}||\theta^{n}||\leq C k^{\alpha+1}.\label{54}
\end{eqnarray}
We let $\epsilon^{n}=V^{n}-U^{n}\equiv u^{N}(t^{n})-U^{n}$. Our aim will be to prove that $\max_{n}||\epsilon^{n}||=O(k^{\alpha})$, which, together with (\ref{dd31}), will give the desired error estimate 
\begin{eqnarray*}
\max_{0\leq n\leq M}||U^{n}-u(t^{n})||\leq C\left(k^{\alpha}+N^{1-\mu}\right).
\end{eqnarray*}
We also let $\epsilon^{n,i}=V^{n,i}-Y^{n,i}, 1\leq i\leq s$, and note, in view of (\ref{53}), (\ref{51}), that
\begin{eqnarray*}
\epsilon^{n+1}=V^{n+1}-U^{n+1}=\theta^{n}+V^{n,s}-Y^{n,s}=\theta^{n}+\epsilon^{n,s}.\label{55}
\end{eqnarray*}
Obviously, cf. section \ref{sec43}, the $Y^{n,i}$ and $V^{n,i}$ exist and satisfy for all $n$ and $1\leq i\leq s$ the $L^{2}$ conservation laws
\begin{eqnarray}
&&||Y^{n,i}||=||U^{n}||=||U^{0}||,\nonumber\\
&&||V^{n,i}||=||V^{n}||=||u^{N}(t^{n})||=||u^{N}(0)||.\label{57}
\end{eqnarray}
In order to bound the $\epsilon^{n,i}$ and $\epsilon^{n}$ in $L^{2}$, the $L^{2}$ bounds of $V^{n,i}$ in (\ref{57}) are not enough. So we first establish in Lemma \ref{lemma51} a bound for $||V^{n,i}||_{1,\infty}$ uniformly in $n$ and $i$. The proof of Lemma \ref{lemma52} follows easily; in it
we show that $\max_{n}||\epsilon^{n}||\leq Ck^{\alpha}$ after establishing estimates of the form
$$\max_{i}||\epsilon^{n,i}||\leq (1+Ck)||\epsilon^{n}||.$$
Finally, in Theorem \ref{Theo51}, we prove the uniqueness of the fully discrete approximations $U^{n}, Y^{n,i}$, and the final error estimate 
\begin{eqnarray*}
\max_{n}||U^{n}-u(t^{n})||\leq C\left(k^{\alpha}+N^{1-\mu}\right).
\end{eqnarray*}
\begin{lemma}
\label{lemma51}
Let $V^{n,i}$ be defined by (\ref{52}). Suppose that $\mu$ is sufficiently large, $k$ is sufficiently small, and that $k=O(N^{-1/2})$. Then
\begin{eqnarray}
\max_{i,n}||V^{n,i}||_{1,\infty}\leq C.\label{58}
\end{eqnarray} 
\end{lemma}

\begin{remark*}
Here and in the sequel we let $\tau^{n,i}=t^{n}+k(b_{1}+b_{2}+\cdots+b_{i}), 1\leq i\leq s$, so that $\tau^{n,s}=t^{n+1}$. Since some of the $b_{i}$ may be negative, and some $\tau^{n,i}$ may exceed $t^{n+1}$, as was remarked in the course of the proof of Proposition \ref{propo43}, it may be necessary to extend the well-posedness of (\ref{dd12}) and the validity of (\ref{dd24}) in temporal intervals of the form $[-l_{1}k,T+l_{2}k]$ for small nonnegative integers $l_{1}, l_{2}$. In such temporal intervals the bounds in (\ref{dd31}) and (\ref{dd34}) obviously hold.
\end{remark*}
\begin{proof}
We break the proof in three steps for ease in reading it.
\begin{itemize}
\item[(i)] \underline{First prove that $\max_{n}||V^{n,1}||_{1,\infty}\leq C$}.
\end{itemize}
We will show that $V^{n,1}$ is close to $u^{N}(\tau^{n,1})$, specifically to $O(k^{3})$ in $L^{2}$, and then use (\ref{dd23}) and (\ref{dd34}) to prove the desired bound. For this purpose we first need the following consistency result for one step of length $kb_{1}$ for the IMR scheme, which is easily established. As before, we denote the values of $u^{N}$ and its derivatives at $t^{n}$ simply by $u^{N}, u_{t}^{N}$, etc.

Define $\zeta^{n,1}\in S_{N}$ by the equation
\begin{eqnarray}
u^{N}(\tau^{n,1})=u^{N}+kb_{1}F\left(\frac{u^{N}(\tau^{n,1})+u^{N}}{2}\right)+\zeta^{n,1}.\label{59}
\end{eqnarray}
Then, as expected, we have, for $\mu$ sufficiently large,
\begin{eqnarray}
\max_{n}||\zeta^{n,1}||\leq Ck^{3}.\label{510}
\end{eqnarray}
To see this, in our setting, we write
\begin{eqnarray}
\zeta^{n,1}=\omega_{1}^{n}-\omega_{2}^{n},\label{511}
\end{eqnarray}
where
\begin{eqnarray*}
\omega_{1}^{n}&=&u^{N}(\tau^{n,1})-u^{N}-\frac{kb_{1}}{2}\left(u_{t}^{N}(\tau^{n,1})+u_{t}^{N}\right),\\
\omega_{2}^{n}&=&kb_{1}P_{N}\left(\frac{f(u^{N}(\tau^{n,1}))_{x}+f(u^{N})_{x}}{2}-f\left(\frac{u^{N}(\tau^{n,1})+u^{N}}{2}\right)_{x}\right).
\end{eqnarray*}
By Taylor's theorem and (\ref{dd34}) we get for $\mu$ sufficiently large
\begin{eqnarray}
||\omega_{1}^{n}||\leq Ck^{3}\max_{t}||\partial_{t}^{3}u^{N}||\leq Ck^{3}.\label{513}
\end{eqnarray}
To estimate $\omega_{2}^{n}$ we write
\begin{eqnarray}
\omega_{2}^{n}=kb_{1}(\rho^{n}-\sigma^{n}),\label{514}
\end{eqnarray}
where
\begin{eqnarray*}
\rho^{n}=P_{N}\left(\frac{f(u^{N}(\tau^{n,1}))_{x}+f(u^{N})_{x}}{2}\right),\quad
\sigma^{n}=P_{N}\left(f\left(\frac{u^{N}(\tau^{n,1})+u^{N}}{2}\right)_{x}\right).
\end{eqnarray*}
For $\rho^{n}$ note that by (\ref{dd24})
\begin{eqnarray*}
\rho^{n}=-\frac{1}{2}\left(u_{t}^{N}(\tau^{n,1})+u_{t}^{N}\right)-\frac{1}{2}\partial_{x}^{3}\left(u^{N}(\tau^{n,1})+u^{N}\right).
\end{eqnarray*}
Hence, by Taylor's theorem, putting $s^{n,1}=\frac{1}{2}(t^{n}+\tau^{n,1})$ we get
\begin{eqnarray}
\rho^{n}=-u_{t}^{N}(s^{n,1})-\partial_{x}^{3}u^{N}(s^{n,1})+\widetilde{\rho}^{n},\label{516}
\end{eqnarray}
where, for $\mu$ sufficiently large, by (\ref{dd34})
\begin{eqnarray}
||\widetilde{\rho}^{n}||\leq Ck^{2}.\label{517}
\end{eqnarray}
To estimate $\sigma^{n}$, let $\eta^{n,1}=\frac{1}{2}\left(u^{N}(\tau^{n,1})+u^{N}\right)-u^{N}(s^{n,1})$ so that, for $\mu$ sufficiently large and by (\ref{dd34})
\begin{eqnarray}
||\eta^{n,1}||_{1}\leq Ck^{2}.\label{518}
\end{eqnarray}
Therefore for $\sigma^{n}$ we have
\begin{eqnarray*}
\sigma^{n}=&=&P_{N}\left(f(u^{N}(s^{n,1})+\eta^{n,1})_{x}\right)\\
&=&P_{N}\left(f(u^{N}(s^{n,1}))_{x}+f(\eta^{n,1})_{x}+(u^{N}(s^{n,1}),\eta^{n,1})_{x}\right),
\end{eqnarray*}
i.~e.
\begin{eqnarray}
\sigma^{n}=P_{N}\left(f(u^{N}(s^{n,1}))_{x}\right)+\widetilde{\sigma}^{n},\label{519}
\end{eqnarray}
where, by (\ref{518}), Sobolev's theorem, and (\ref{dd34}), for $\mu$ sufficiently large,
\begin{eqnarray}
||\widetilde{\sigma}^{n}||\leq Ck^{2}.\label{520}
\end{eqnarray}
Therefore, by (\ref{514}), (\ref{516}), (\ref{519}) and (\ref{dd36}) we get
\begin{eqnarray*}
\omega_{2}^{n}&=&kb_{1}\left(-u_{t}^{N}(s^{n,1})-\partial_{x}^{3}u^{N}(s^{n,1})-P_{N}\left(f(u^{N}(s^{n,1}))_{x}\right)+\widetilde{\rho}^{n}-\widetilde{\sigma}^{n}\right)\\
&=&kb_{1}(\widetilde{\rho}^{n}-\widetilde{\sigma}^{n}),
\end{eqnarray*}
and by (\ref{517}) and (\ref{520})
\begin{eqnarray*}
||\omega_{2}^{n}||\leq Ck^{3},\label{521}
\end{eqnarray*}
which yields (\ref{510}), in view of (\ref{511}) and (\ref{513}).

We now proceed to bound $V^{n,1}$ in the $||\cdot||_{1,\infty}$ norm. By (\ref{52}) for $i=1$ and (\ref{59}), (since $V^{n}=u^{N}$), we obtain
\begin{eqnarray*}
V^{n,1}-u^{N}(\tau^{n,1})
&=&kb_{1}\left(F\left(\frac{V^{n,1}+u^{N}}{2}\right)-F\left(\frac{u^{N}(\tau^{n,1})+u^{N}}{2}\right)\right)-\zeta^{n,1}.
\end{eqnarray*}
Therefore, by integration by parts we see that
\begin{eqnarray}
||V^{n,1}-u^{N}(\tau^{n,1})||^{2}
&=&kb_{1}\left(f\left(\frac{V^{n,1}+u^{N}}{2}\right)_{x}-f\left(\frac{u^{N}(\tau^{n,1})+u^{N}}{2}\right)_{x},V^{n,1}-u^{N}(\tau^{n,1})\right)\nonumber\\
&&-\left(\zeta^{n,1},V^{n,1}-u^{N}(\tau^{n,1})\right).\label{522}
\end{eqnarray}
Now, by integration by parts, and (\ref{dd34}), for $\mu$ sufficiently large, we see that
\begin{eqnarray*}
&&\left|\left(f\left(\frac{V^{n,1}+u^{N}}{2}\right)_{x}-f\left(\frac{u^{N}(\tau^{n,1})+u^{N}}{2}\right)_{x},V^{n,1}-u^{N}(\tau^{n,1})\right)\right|\\
&=&\left|\left(f\left(\frac{u^{N}(\tau^{n,1})+u^{N}}{2}+
\frac{V^{n,1}-u^{N}(\tau^{n,1})}{2}\right)_{x}-f\left(\frac{u^{N}(\tau^{n,1})+u^{N}}{2}\right)_{x},\right.\right.\\
&&\left.\left.{V^{n,1}-u^{N}(\tau^{n,1})}\right)\right|\\
&=&\left|\left(\left[\left(\frac{u^{N}(\tau^{n,1})+u^{N}}{2}\right)\left(
\frac{V^{n,1}-u^{N}(\tau^{n,1})}{2}\right)\right]_{x},
{V^{n,1}-u^{N}(\tau^{n,1})}\right)\right|\\
&\leq &C||u^{N}(\tau^{n,1})+u^{N}||_{1,\infty}||V^{n,1}-u^{N}(\tau^{n,1})||^{2}\\
&\leq &C||V^{n,1}-u^{N}(\tau^{n,1})||^{2}.
\end{eqnarray*}
We conclude by (\ref{522}), and (\ref{510}), for $k$ sufficiently small, that
\begin{eqnarray}
||V^{n,1}-u^{N}(\tau^{n,1})||\leq Ck^{3}.\label{523}
\end{eqnarray}
Therefore, by the above, (\ref{dd23}), and (\ref{dd34}), for $\mu$ sufficiently large, we get
\begin{eqnarray}
||V^{n,1}||_{1,\infty}&\leq &||V^{n,1}-u^{N}(\tau^{n,1})||_{1,\infty}+||u^{N}(\tau^{n,1})||_{1,\infty}\nonumber\\
&\leq &Ck^{3}N^{3/2}+C\leq C,\label{524}
\end{eqnarray}
using the mesh condition $k=O\left(N^{-1/2}\right)$.
\begin{itemize}
\item[(ii)] \underline{Prove now that $\max_{n}||V^{n,2}||_{1,\infty}\leq C$}.
\end{itemize}
We will follow the same general plan as in (i). We let $\zeta^{n,2}$ be the local temporal error of the scheme during the substep $\tau^{n,1}\mapsto\tau^{n,2}$, i.~e. define it by the equation
\begin{eqnarray}
u^{N}(\tau^{n,2})=u^{N}(\tau^{n,1})+kb_{2}F\left(\frac{u^{N}(\tau^{n,2})+u^{N}(\tau^{n,1})}{2}\right)+\zeta^{n,2}.\label{525}
\end{eqnarray}
Then, we may prove as in (i), {\it mutatis mutandis} that
\begin{eqnarray}
\max_{n}||\zeta^{n,2}||\leq Ck^{3}.\label{526}
\end{eqnarray}
By (\ref{52}) for $i=2$ and (\ref{525}) we have
\begin{eqnarray}
V^{n,2}-u^{N}(\tau^{n,2})&=&V^{n,1}-u^{N}(\tau^{n,1})\nonumber\\
&&+kb_{2}\left(F\left(\frac{V^{n,2}+V^{n,1}}{2}\right)-F\left(\frac{u^{N}(\tau^{n,2})+u^{N}(\tau^{n,1})}{2}\right)\right)\nonumber\\
&&-\zeta^{n,2}.\label{527}
\end{eqnarray}
In order to simplify a bit the algebra we define $\chi_{j}\in S_{N}, 1\leq j\leq 4$, as
\begin{eqnarray*}
&&\chi_{1}=V^{n,2}-u^{N}(\tau^{n,2}),\quad
\chi_{2}=V^{n,1}-u^{N}(\tau^{n,1}),\\
&&\chi_{3}=\frac{V^{n,2}+V^{n,1}}{2},\quad
\chi_{4}=\frac{u^{N}(\tau^{n,2})+u^{N}(\tau^{n,1})}{2}.
\end{eqnarray*}
Then, (\ref{527}) is written as
\begin{eqnarray*}
\chi_{1}-\chi_{2}=kb_{2}\left(F\left(\chi_{3}\right)-F\left(\chi_{4}\right)\right)-\zeta^{n,2}.\label{528a}
\end{eqnarray*}
Take $L^{2}$ inner products in the above with $\frac{\chi_{1}+\chi_{2}}{2}$, noting that $\frac{\chi_{1}+\chi_{2}}{2}=\chi_{3}-\chi_{4}$ and using integration by parts, and get
\begin{eqnarray}
\frac{1}{2}\left(||\chi_{1}||^{2}-||\chi_{2}||^{2}\right)&=&
kb_{2}\left(f\left(\chi_{4}+\frac{\chi_{1}+\chi_{2}}{2}\right)_{x}-f\left(\chi_{4}\right)_{x},\frac{\chi_{1}+\chi_{2}}{2}\right)\nonumber\\
&&-\left(\zeta^{n,2},\frac{\chi_{1}+\chi_{2}}{2}\right)\label{528}
\end{eqnarray}
Now, by integration by parts and (\ref{dd34}), for $\mu$ sufficiently large, we see that
\begin{eqnarray}
\left|\left(f\left(\chi_{4}+\frac{\chi_{1}+\chi_{2}}{2}\right)_{x}-f\left(\chi_{4}\right)_{x},\frac{\chi_{1}+\chi_{2}}{2}\right)\right|&=&
\left|\left(\left(\chi_{4}\left(\frac{\chi_{1}+\chi_{2}}{2}\right)\right)_{x},\frac{\chi_{1}+\chi_{2}}{2}\right)\right|\nonumber\\
&&\leq  C||\chi_{4}||_{1,\infty}
||\chi_{1}+\chi_{2}||^{2}\nonumber\\
&\leq &C||\chi_{1}+\chi_{2}||^{2},\label{529}
\end{eqnarray}
and (\ref{526}), (\ref{528}), (\ref{529}) yield
\begin{eqnarray*}
\frac{1}{2}\left(||\chi_{1}||^{2}-||\chi_{2}||^{2}\right)\leq Ck\left(||\chi_{1}||+||\chi_{2}||\right)^{2}+Ck^{3}||\left(||\chi_{1}||+||\chi_{2}||\right),
\end{eqnarray*}
i.~e.
\begin{eqnarray*}
||\chi_{1}||-||\chi_{2}||\leq Ck\left(||\chi_{1}||+||\chi_{2}||\right)+Ck^{3},
\end{eqnarray*}
from which, if we recall the definition of $\chi_{1}$ and $\chi_{2}$, it follows for $k$ sufficiently small, that

\begin{eqnarray*}
||V^{n,2}-u^{N}(\tau^{n,2})||
\leq C||V^{n,1}-u^{N}(\tau^{n,1})||+Ck^{3}.
\end{eqnarray*}
Therefore, by (\ref{523})
\begin{eqnarray}
||V^{n,2}-u^{N}(\tau^{n,2})||\leq Ck^{3},\label{532}
\end{eqnarray}
from which, as in the derivation of (\ref{524}), we get, for $\mu$ sufficiently large, since $k=O(N^{-1/2})$, that
\begin{eqnarray*}
||V^{n,2}||_{1,\infty}\leq C.\label{533}
\end{eqnarray*}
Hence, the proof of (ii) is complete.
\begin{itemize}
\item[(iii)] \underline{Prove that $||V^{n,i}||_{1,\infty}\leq C, 3\leq i\leq s$}.
\end{itemize}
The bounds $||V^{n,i}-u^{N}(\tau^{n,i})||\leq Ck^{3}$ implying that
$||V^{n,i}||_{1,\infty}\leq C, 3\leq i\leq s$, are obtained entirely analogously, as in step (ii) above, and their proof is omitted. We conclude that (\ref{58}) holds.
\end{proof}

\begin{lemma}
\label{lemma52}
Let $\epsilon^{n}=V^{n}-U^{n}$, where $V^{n}=u^{N}$, and $U^{n}$ is the fully discrete approximation, defined by (\ref{51}). We assume the smoothness of $u$ and the mesh condition stated in Lemma \ref{lemma51} and we suppose that the temporal local error estimate (\ref{54}) holds. Then
\begin{eqnarray}
\max_{n}||\epsilon^{n}||\leq Ck^{\alpha}.\label{534}
\end{eqnarray}
\end{lemma}
\begin{proof}
We use throughout the notation introduced in the beginning of the section.
We first estimate $\epsilon^{n,1}=V^{n,1}-Y^{n,1}$ in terms of $\epsilon^{n}$.
Since 
\begin{eqnarray*}
\epsilon^{n,1}=\epsilon^{n}+kb_{1}\left(F\left(\frac{V^{n,1}+V^{n}}{2}\right)-F\left(\frac{Y^{n,1}+U^{n}}{2}\right)\right),
\end{eqnarray*}
we have
\begin{eqnarray*}
\epsilon^{n,1}-\epsilon^{n}=kb_{1}\left(F\left(\frac{V^{n,1}+V^{n}}{2}\right)-F\left(\frac{V^{n,1}+V^{n}}{2}-\frac{\epsilon^{n,1}+\epsilon^{n}}{2}\right)\right).
\end{eqnarray*}
Taking $L^{2}$ inner products in this equation with $\frac{\epsilon^{n,1}+\epsilon^{n}}{2}$ we obtain, by integration by parts
\begin{eqnarray*}
\frac{1}{2}\left(||\epsilon^{n,1}||^{2}-||\epsilon^{n}||^{2}\right)
&=&-kb_{1}\left(f\left(\frac{V^{n,1}+V^{n}}{2}-\frac{\epsilon^{n,1}+\epsilon^{n}}{2}\right)_{x}-f\left(\frac{V^{n,1}+V^{n}}{2}\right)_{x},\frac{\epsilon^{n,1}+\epsilon^{n}}{2}\right).
\end{eqnarray*}
Therefore, using integration by parts again, we get
\begin{eqnarray*}
||\epsilon^{n,1}||^{2}-||\epsilon^{n}||^{2}\leq
Ck ||V^{n,1}+V^{n}||_{1,\infty}
||\epsilon^{n,1}+\epsilon^{n}||^{2},
\end{eqnarray*}
from which, taking into account (\ref{58}) and (\ref{dd34}) , it follows that
\begin{eqnarray*}
||\epsilon^{n,1}||-||\epsilon^{n}||\leq
Ck 
\left(||\epsilon^{n,1}||+||\epsilon^{n}||\right),\label{535}
\end{eqnarray*}
for all $n$. Hence, for $k$ sufficiently small, for all $n$ it holds that
\begin{eqnarray}
||\epsilon^{n,1}||\leq
(1+Ck) ||\epsilon^{n}||.\label{537}
\end{eqnarray}
We get similarly that
\begin{eqnarray}
\max_{i}||\epsilon^{n,i}||\leq
(1+Ck) ||\epsilon^{n}||.\label{538}
\end{eqnarray}
This may be seen as follows: Since in view of (\ref{58}), as previously, there holds that
\begin{eqnarray*}
||\epsilon^{n,2}||-||\epsilon^{n,1}||\leq
Ck 
\left(||\epsilon^{n,2}||+||\epsilon^{n,1}||\right),
\end{eqnarray*}
we obtain by (\ref{537}) that $||\epsilon^{n,2}||\leq (1+Ck) ||\epsilon^{n}||$. The general case (\ref{538}) follows inductively.

Recall by (\ref{51}) and (\ref{53}) that $\epsilon^{n+1}=V^{n+1}-U^{n+1}=V^{n+1}-Y^{n,s}=V^{n,s}-Y^{n,s}+\theta^{n}=\epsilon^{n,s}+\theta^{n}$. Therefore, by (\ref{538}), for all $n$ we have
\begin{eqnarray*}
||\epsilon^{n+1}||\leq
(1+Ck) ||\epsilon^{n}||+||\theta^{n}||,
\end{eqnarray*}
from which, by the discrete Gronwall inequality, since $\epsilon^{0}=0$, and the hypothesis (\ref{54}), we conclude that (\ref{534}) holds.
\end{proof}
We now state and prove the main error estimate for our fully discrete method.
\begin{theorem}
\label{Theo51}
Suppose that $\mu$ is sufficiently large, that (\ref{54}) holds for some $\alpha\geq 1$ and suppose that $kN$ is sufficiently small. Then the fully discrete scheme (\ref{51}) has for all $n$ a unique solution $U^{n}$ such that
\begin{eqnarray}
\max_{n}||U^{n}-u(t^{n})||\leq C(k^{\alpha}+N^{1-\mu}).\label{543}
\end{eqnarray} 
\end{theorem}
\begin{proof}
Let $U^{n}$ be a solution of (\ref{51}). Then
\begin{eqnarray*}
||U^{n}-u(t^{n})||&\leq &||U^{n}-u^{N}(t^{n})||+||u^{N}(t^{n})-u(t^{n})||
=||\epsilon^{n}||+||u^{N}(t^{n})-u(t^{n})||,
\end{eqnarray*}
and (\ref{543}) follows from (\ref{537}) and (\ref{dd31}).

In order to prove the uniqueness of $U^{n}$ we have to verify the hypotheses of Lemma \ref{lemma41}. Note that it follows from (\ref{534}), (\ref{dd34}), (\ref{dd23}), and our mesh condition that
for all $n$, $|U^{n}|_{\infty}\leq |\epsilon^{n}|_{\infty}+|u^{N}|_{\infty}\leq Ck^{\alpha}N^{1/2}+C\leq R_{1}$, for some constant $R_{1}$, independent of $k, N$. In addition, for all $i$ and $n$, by (\ref{dd23}), (\ref{58}), (\ref{538}), (\ref{534}), and our mesh condition, we see that
\begin{eqnarray*}
|Y^{n,i}|_{\infty}&\leq &|\epsilon^{n,i}|_{\infty}+|V^{n,i}|_{\infty}\leq CN^{1/2}||\epsilon^{n,i}||+C\\
&\leq & CN^{1/2}||\epsilon^{n}||+C\leq CN^{1/2}k^{\alpha}+C\leq R_{2},
\end{eqnarray*}
for some constant $R_{2}$ independent of $k$ and $N$.
If $R$ is taken as $\max(R_{1},R_{2})$ by Lemma \ref{lemma41} we have uniqueness of $U^{n}=Y^{n,s}$ if $kN$ is sufficiently small as we have assumed.
\end{proof}
\begin{remark}
Since the fully discrete scheme (\ref{51}) is written as a sequence of IMR steps, its implementation is quite straightforward, as the attendant nonlinear systems are decoupled and may each be solved by an iterative scheme. 

Indeed, suppose that, for some $n$ and $i\geq 1$ $Y^{n,i-1}$ is known. Then if $Z^{*}=\frac{1}{2}\left(Y^{n,i-1}+Y^{n,i}\right)$ it follows that $Z^{*}\in S_{N}$ satisfies
\begin{eqnarray}
Z^{*}=Y^{n,i-1}+\frac{kb_{i}}{2}F(Z^{*}).\label{537r}
\end{eqnarray}
Suppose that the hypotheses of Theorem \ref{Theo51} hold. Then $Z^{*}$ is unique, and if it is known, $Y^{n,i}$ may be computed as $Y^{n,i}=2Z^{*}-Y^{n,i-1}$.

In order to approximate $Z^{*}$, consider the following simple iterative scheme. For $\nu=0,1,2,\ldots$, seek $Z_{\nu}\in S_{N}$, such that
\begin{eqnarray}
Z_{0}&=&Y^{n,i-1},\nonumber\\
\left(I+\frac{kb_{i}}{2}\partial_{x}^{3}\right)Z_{\nu+1}&=&Y^{n,i-1}-\frac{kb_{i}}{2}f(Z_{\nu})_{x},\; \nu=0,1,2,\ldots\label{538r}
\end{eqnarray}
Given $Z_{\nu}$, $Z_{\nu+1}$ satisfies a linear system of equations. The associated homogeneous system clearly has only the trivial solution; hence $Z_{\nu+1}$ is uniquely defined and its Fourier coefficients may be readily computed.

In order to prove the convergence of this scheme, substract the equation defining $Z_{\nu+1}$ in (\ref{538r}) from (\ref{537r}) and get
\begin{eqnarray*}
\left(I+\frac{kb_{i}}{2}\partial_{x}^{3}\right)(Z^{*}-Z_{\nu+1})=-\frac{kb_{i}}{2}\left(f(Z^{*})_{x}-f(Z_{\nu})_{x}\right).
\end{eqnarray*}
Taking the $L^{2}$ inner product of both sides of this equation with $Z^{*}-Z_{\nu+1}$ we see by periodicity that
\begin{eqnarray*}
||Z^{*}-Z_{\nu+1}||^{2}=-\frac{kb_{i}}{2}\left(f(Z^{*})_{x}-f(Z_{\nu})_{x},Z^{*}-Z_{\nu+1}\right),
\end{eqnarray*}
which implies, in view of (\ref{dd23}) and the definition of $f$, that
\begin{eqnarray}
||Z^{*}-Z_{\nu+1}||\leq\frac{C_{0}}{4}|b_{i}|kN |Z^{*}-Z_{\nu+1}|_{\infty}||Z^{*}-Z_{\nu}||.\label{539r}
\end{eqnarray}
Let, for all $n$ and $i$, $|Y^{n,i}|_{\infty}\leq R$. (From the last part of the proof of Theorem \ref{Theo51} we infer that such a constant $R$ exists and is independent of $k$ and $N$.) Then it follows that $|Z^{*}|_{\infty}\leq R$.

We will prove now that for $\nu=0,1,2,\ldots$
\begin{eqnarray}
|Z_{\nu}|_{\infty}\leq R+1.\label{540r}
\end{eqnarray}
Obviously, (\ref{540r}) holds for $\nu=0$. Let now $\nu^{*}\geq 0$ be the maximal integer for which
\begin{eqnarray}
|Z_{\nu}|_{\infty}\leq R+1,\;{\rm for}\; 0\leq\nu\leq\nu^{*}.\label{541r}
\end{eqnarray}
Then, letting $\Gamma=\frac{C_{0}}{2}k|b_{i}|N(R+1)$ and assuming that $kN$ is sufficiently small so that $\Gamma<1$, we have from (\ref{539r}) for $0\leq\nu\leq\nu^{*}$
\begin{eqnarray*}
||Z^{*}-Z_{\nu+1}||\leq\Gamma ||Z^{*}-Z_{\nu}||,
\end{eqnarray*}
and therefore 
\begin{eqnarray}
||Z^{*}-Z_{\nu+1}||\leq\Gamma^{\nu+1} ||Z^{*}-Z_{0}||,\;\;
0\leq\nu\leq\nu^{*}.\label{542r}
\end{eqnarray}
Now, by (\ref{dd23})
\begin{eqnarray*}
|Z_{\nu+1}|_{\infty}\leq |Z^{*}-Z_{\nu+1}|_{\infty}+|Z^{*}|_{\infty}\leq C_{0}N^{1/2}||Z^{*}-Z_{\nu+1}||+R.
\end{eqnarray*}
Hence, using (\ref{542r}) we have
\begin{eqnarray}
|Z_{\nu+1}|_{\infty}\leq C_{0}N^{1/2}\Gamma^{\nu+1} ||Z^{*}-Z_{0}||+R,\;0\leq\nu\leq\nu^{*}.\label{543r} 
\end{eqnarray}
Now, employing the notation introduced in the beginning of the section, we have by the triangle inequality
\begin{eqnarray*}
||Z^{*}-Z_{0}||&=&\left|\left|\frac{1}{2}\left(Y^{n,i}+Y^{n,i-1}\right)-Y^{n,i-1}\right|\right|=\frac{1}{2}||Y^{n,i}-Y^{n,i-1}||\\
&\leq & \frac{1}{2}\left(||\epsilon^{n,i}||+||\epsilon^{n,i-1}||+||V^{n,i}-u^{N}(\tau^{n,i})||\right.\\
&&\left.+||V^{n,i-1}-u^{N}(\tau^{n,i-1})||+||u^{N}(\tau^{n,i})-u^{N}(\tau^{n,i-1})||\right)\\
&\leq & C(k^{\alpha}+k^{3}+k),
\end{eqnarray*}
where, in the last inequality, we was made of (\ref{538}), (\ref{534}), (\ref{523}), (\ref{532}), the remark in (iii) of Lemma \ref{lemma51}, Taylor's theorem, and (\ref{dd34}) for $\mu$ sufficiently large. Therefore, for some constant $C_{1}$, independent of $k$ and $N$, we have
\begin{eqnarray}
||Z^{*}-Z_{0}||\leq C_{1} k.\label{544r}
\end{eqnarray}
We conclude by (\ref{543r}) and (\ref{544r}), in view of our hypothesis on $kN$, that $k$ may be chosen sufficiently small so that for $0\leq \nu\leq\nu^{*}$ it holds that $|Z_{\nu+1}|_{\infty}\leq R+1/2$. Therefore, $\nu^{*}$ was not maximal in (\ref{541r}), and we may continue the argument for all $\nu$. Hence (\ref{542r}) holds for all $\nu$ and gives that $Z_{\nu}$ converges to $Z^{*}$ as $\nu\rightarrow\infty$; from (\ref{542r}) and (\ref{544r}) we have the estimate $||Z_{\nu}-Z^{*}||\leq C\Gamma^{\nu}k, \nu\geq 0$. Therefore, for $\nu=O(abs(\log k))$ we may bound $||Z^{*}-Z_{\nu}||$ by a constant times a sufficiently large power of $k$. We will not analyze here the stability of the overall time-marching scheme.
\end{remark}
\section{Conclusions and extensions}
\label{sec6}
In this paper we analyzed a high-order accurate fully discrete scheme for the periodic ivp for the KdV equation. The problem was discretized in space by the standard Fourier-Galerkin spectral method. For the temporal discretization we used  a diagonally implicit Runge-Kutta scheme of composition type with $s$ stages, cf. \cite{Yoshida1990,HairerLW2004}, effected by $s$ steps of the IMR method. This type of schemes are not A-stable but they are symplectic; hence they are unconditionally $L^{2}-$conservative for the periodic ivp and semidiscretization at hand. They are also easy to implement. We proved that the local temporal error of the scheme with $s=3$ stages applied to the semidiscrete equations is of $O(k^{5})$ in $L^{2}$, where $k$ is the time step, under the hypothesis that that the solution of the periodic ivp belongs to the periodic Sobolev space $H^{\mu}$ for $\mu$ sufficiently  large and that $k=O(N^{-1})$, where $N$ is the order of the trigonometric polynomials used in the semidiscretization. We also proved that if $kN$ is sufficiently small the fully discrete scheme has a unique solution and satisfies an $L^{2}$ error estimate of $O(k^{\alpha}+N^{1-\mu})$, provided the local temporal error is of $O(k^{\alpha+1})$ in $L^{2}$ and if $\mu$ is sufficiently large. So, for the particular scheme with $s=3$ stages (first used in computations for solving the KdV in \cite{FrutosS1992}), the resulting error estimate is of $O(k^{4}+N^{1-\mu})$.

The Runge-Kutta scheme considered in this paper may be used to discretize in the temporal variable other conservative ivp's for pde's that model one-way propagation of nonlinear dispersive waves. For example, in \cite{DDM2019} the three-stage, fourth-order accurate scheme, coupled with a spectral discretization in space, was used for approximating the solution of the periodic ivp for the generalized Benjamin equation. This nonlinear pde, introduced in \cite{BonaC1998} as a generalization of the KdV, the Benjamin-Ono, and the Benjamin equation, is of the form
\begin{eqnarray}
u_{t}-\mathcal{L}u_{x}+f(u)_{x}=0,\label{61}
\end{eqnarray}
where $\mathcal{L}$ is the linear, nonlocal, pseudodifferential operator with Fourier symbol
\begin{eqnarray*}
\widehat{\mathcal{L}u}(\xi)=l(\xi)\widehat{u}(\xi)=(\delta |\xi|^{2m}-\gamma
|\xi|^{2r})\widehat{u}(\xi),\quad \xi\in\mathbb{R}, \label{intro2}
\end{eqnarray*}
where $m\geq 1$ is an integer, $0\leq r<m, \gamma\geq 0, \delta>0$, $\widehat{u}(\xi)$ denotes the Fourier transform of $u$ at $\xi$, and the nonlinear term $f$ is given by
\begin{eqnarray*}
f(u)=\frac{u^{q+1}}{q+1},
\end{eqnarray*}
with $q\geq 1$ integer. The Cauchy problem for (\ref{61}) has been shown to be locally well-posed in $H^{s}(\mathbb{R})$ for $s\geq 1$, see e.~g. \cite{LinaresS2005}, and globally well-posed if $q=2$ or $3$. The equation possesses solitary-wave solutions, cf. \cite{BonaC1998}; the numerical study in \cite{DDM2019} was focused on describing their generation, interactions, and stability. In the case of the periodic ivp for (\ref{61}), the standard Fourier-Galerkin semidiscrete approximation $u^{N}$ may be shown to possess an $L^{2}-$error estimate of the form $||u-u^{N}||\leq CN^{1-\mu}$ if $u\in H^{\mu}, \mu\geq 5/2$, and satisfy (\ref{dd34}) if $\mu$ is sufficiently large. If we define now, for $v\in S_{N}, F(v)=\mathcal{L}v_{x}-P_{N}f(v)_{x}$, so that $(F(v),v)=0$ for $v\in S_{N}$, it may be seen that the proof of existence of solutions and of the $L^{2}-$conservation property of the fully discrete scheme proceeds as in section \ref{sec43}. The study of the local temporal error of the scheme with $s=3$ stages may be done along the lines of Proposition \ref{propo43}. An analog of Theorem \ref{Theo51} holds as well. It may be proved that the solution of the $s-$stage, fully discrete scheme is unique and satisfies (\ref{543}) {\it mutatis mutandis}, under the assumptions that the solution of the periodic ivp is sufficiently smooth, and that $kN$ is sufficiently small if $q=1,2$ or $3$, and $kN^{\frac{q-1}{2}}$ is sufficiently small if $q\geq 4$. The general plan of the proof is that of Theorem \ref{Theo51} but, as expected, considerable technical complications enter the picture due the generalized nonlinear term.

\section*{Acknowledgement}
The authors would like to acknowledge travel support, that made possible this collaboration, from the Institute of Applied and Computational Mathematics of FORTH and the Institute of Mathematics (IMUVA) of the University of Valladolid.
\bibliographystyle{siamplain}
\bibliography{references_DD}
\end{document}

%% file: tmp_KdV_34_v3_header.tex
\title{A high order fully discrete scheme for the Korteweg-de vries equation with a time-stepping procedure of Runge-Kutta-composition type
}


\author{Vassilios A. Dougalis\thanks{Mathematics Department, University of Athens, 15784
Zographou, Greece and
Institute of Applied and Computational
Mathematics, FO.R.T.H., 71110 Heraklion, Greece
(\email{doug@math.uoa.gr}).}
\and Angel Dur\'an\thanks{Applied Mathematics Department,  University of
Valladolid, 47011 Valladolid, Spain and Institute of Mathematics of the University of Valladolid (IMUVA) Paseo de Belen S/N 47011 Valladolid, Spain (\email{angel@mac.uva.es})}}

\headers{A high order fully discrete scheme for the KdV}{V. A. Dougalis, and A. Dur\'an}

%% file: tmp_KdV_34_v3_abstract.tex
\begin{abstract}
We consider the periodic initial-value problem for the Korteweg-de Vries equation that we discretize in space by a spectral Fourier-Galerkin method and in time by an implicit, high order, Runge-Kutta scheme of composition type based on the implicit midpoint rule. We prove $L^{2}$ error estimates for the resulting semidiscrete and the fully discrete approximations.
\end{abstract}
\begin{keywords}
Korteweg-de Vries equation, spectral method, Runge-Kutta-Composition methods, error estimates
\end{keywords}
\begin{AMS}
65M70, 65M12,65L06
\end{AMS}